\documentclass[11pt]{amsart}
\usepackage[margin=1.5in]{geometry}
\usepackage{cite}
\usepackage[utf8]{inputenc}
\usepackage[T1]{fontenc}
\usepackage[english]{babel}
\usepackage{amsmath,amsfonts,amsthm, mathrsfs}
\usepackage{hyperref}
\usepackage{tikz-cd}

\theoremstyle{plain}
\newtheorem{theorem}{Theorem}[section]
\newtheorem{corollary}[theorem]{Corollary}
\newtheorem{lemma}[theorem]{Lemma}
\newtheorem{proposition}[theorem]{Proposition}

\theoremstyle{definition}

\newtheorem{remark}[theorem]{Remark}
\numberwithin{equation}{section}
\numberwithin{table}{section}
\DeclareMathOperator{\des}{des}
\DeclareMathOperator{\E}{E}
\DeclareMathOperator{\Prob}{P}
\DeclareMathOperator{\Var}{Var}
\DeclareMathOperator{\Cov}{Cov}

\title[CLT for Descents in Mallows permutations]{A central limit theorem for descents of a Mallows permutation and its inverse}

\author{Jimmy He}
\address{Department of Mathematics, Stanford University, Stanford, CA  94305}
\email{jimmyhe@stanford.edu}
\begin{document}
\maketitle
\begin{abstract}
This paper studies the asymptotic distribution of descents $\des(w)$ in a permutation $w$, and its inverse, distributed according to the Mallows measure. The Mallows measure is a non-uniform probability measure on permutations introduced to study ranked data. Under this measure, permutations are weighted according to the number of inversions they contain, with the weighting controlled by a parameter $q$. The main results are a Berry-Esseen theorem for $\des(w)+\des(w^{-1})$ as well as a joint central limit theorem for $(\des(w),\des(w^{-1}))$ to a bivariate normal with a non-trivial correlation depending on $q$. The proof uses Stein's method with size-bias coupling along with a regenerative process associated to the Mallows measure.
\end{abstract}

\section{Introduction}
Much is known about various statistics of a uniformly random permutation. More recently, there has been interest in studying non-uniform random permutations as well, with a natural problem being to take a well-studied statistic for a uniformly random permutation, and study it under a different distribution. Many non-uniform permutations have been studied (for example, the Ewens distribution, spatial random permutations, and so on) but this paper focuses on the Mallows distribution.

The Mallows distribution was introduced by Mallows to study non-uniform ranked data \cite{M57}. It is perhaps the most widely used non-uniform distribution on permutations in applied statistics, see \cite{M16b} or \cite{T19} for discussion of the statistical uses for Mallows permutations. Thus, understanding the behaviour of features of Mallows permutations is an important problem. They have also seen applications to the study of one-dependent processes \cite{HHL20} and stable matchings \cite{AHHL18}.

The Mallows measure on the symmetric group $S_n$ of parameter $q>0$, denoted $\mu_q$, is defined by taking $\mu_q(w)$ proportional to $q^{l(w)}$ where $l(w)$ denotes the number of inversions in $w\in S_n$. When $q=1$, this is simply the uniform distribution, and when $q\in (0,1)$, the random permutation is concentrated around the identity.

Descents are one of the most well-understood statistics of a random permutation. It has long been known that the number of descents in a uniformly random permutation is asymptotically Gaussian (see for example \cite{F04}), and that the same is true under the Mallows distribution \cite{BDF10}. 

A more interesting problem is to study the joint distribution of the number of descents in $w$ and in $w^{-1}$ (the distribution of $w^{-1}$ is the same as that of $w$). This was first studied by Vatutin \cite{V96} who showed that they are asymptotically uncorrelated Gaussian random variables. Chatterjee and Diaconis \cite{CD17} gave another proof using Stein's method. This paper extends these results to the Mallows distribution.

As part of the proof, results are also obtained for the sum of the number of descents in $w$ and $w^{-1}$, which is known as the two-sided descent statistic. This has been well studied in the uniform case \cite{V96, CD17, O19}, where it was introduced by Chatterjee and Diaconis to study metrics on the symmetric group \cite{CD17}.

\subsection{Main results}
Let $\des(w)$ denote the number of descents in a permutation $w$. Theorem \ref{thm: main theorem 1} gives a central limit theorem for $\des(w)+\des(w^{-1})$ as long as the variance goes to infinity.
\begin{theorem}
\label{thm: main theorem 1}
Let $w\in S_n$ for $n\geq 2$ be Mallows distributed with parameter $q\in (0,\infty)$ and let $\mu$ and $\sigma^2$ denote the mean and variance of $\des(w)+\des(w^{-1})$. Let $Z$ denote a standard normal random variable. Then for all piecewise continuously differentiable functions $h:\mathbf{R}\to \mathbf{R}$,
\begin{equation*}
\begin{split}
    &\left|\E h\left(\frac{\des(w)+\des(w^{-1})-\mu}{\sigma}\right)-\E h(Z)\right|
    \\\leq &\left(331\|h\|_\infty+167\|h'\|_\infty\max(q^{-1/2},q^{1/2})\right)(n-1)^{-\frac{1}{2}}.
\end{split}
\end{equation*}
\end{theorem}
The regime where the central limit theorem is shown to hold is optimal because when $q=O(n^{-1})$ (or $q^{-1}=O(n^{-1})$) then the variance is bounded. As a complement to the central limit theorem, it is shown that when $qn\to \lambda>0$, then $\des(w)+\des(w^{-1})$ converges to $2N$ where $N$ is Poisson with parameter $\lambda$, see Proposition \ref{prop: poisson behaviour}.

The proof of Theorem \ref{thm: main theorem 1} uses a size-bias coupling and Stein's method. A size-bias coupling is constructed, which gives an upper bound on the distance to a standard normal by a version of Stein's method due to Goldstein and Rinott \cite{GR96}. There are two error terms, with the main difficulty being the variance term. The variance term is controlled by decomposing the random variables and explicitly bounding various covariance terms that are obtained.

The coupling is based on a size-bias coupling of Goldstein \cite{G05} developed to study pattern occurrences within random permutations, and a modification of this was previously used by Conger and Viswanath to study descents of permutations of multisets \cite{CV07}.

This proof is new even for the uniform case. There are now many different approaches for the uniform case, including generating functions \cite{V96}, Stein's method with interaction graphs \cite{CD17} and martingales \cite{O19}, but none seem to extend to the Mallows distribution. 

F\'eray has developed the method of weighted dependency graphs which is able to handle weak dependence \cite{F18}. It is unclear to the author whether these techniques could be applied to this problem.

If $q$ is fixed, the joint convergence of $\des(w)$ and $\des(w^{-1})$ to a bivariate normal with non-trivial correlation can be shown.
\begin{theorem}
\label{thm: bivariate limit}
Fix $q\in(0,\infty)$. Then the limiting asymptotic correlation
\begin{equation*}
    \rho=\lim_{n\to \infty}\frac{\Cov(\des(w),\des(w^{-1}))}{\Var(\des(w))}
\end{equation*}
exists. Moreover, if $q\neq 1$ then $0<\rho<1$. 

Let $w_n\in S_n$ is Mallows distributed with parameter $q$, $\mu_n=\E(\des(w_n))$, $\sigma_n^2=\Var(\des(w_n))$ and
\begin{equation*}
    (Z_1,Z_2)\sim N\left(0,\left(\begin{array}{cc}
         1&\rho  \\
         \rho&1 
    \end{array}\right)\right).
\end{equation*}
Then
\begin{equation*}
    \left(\frac{\des(w_n)-\mu_n}{\sigma_n},\frac{\des(w_n^{-1})-\mu_n}{\sigma_n}\right)\xrightarrow{d}(Z_1,Z_2).
\end{equation*}

In addition, if $q=q_n$ varies with $n$ with $q_n\to 0$ (or $q_n\to \infty$), then the above applies with $\rho=1$ as long as $nq_n\to \infty$ (or $n/q_n\to\infty$). If $q_n\to 1$, then the above applies with $\rho=0$.
\end{theorem}

Theorem \ref{thm: bivariate limit} follows from the proof of Theorem \ref{thm: main theorem 1} and the Cram\'er-Wold device after the asymptotic correlation $\rho$ is shown to exist when $q$ is fixed. The computation of the asymptotic correlation uses the Mallows process, a regenerative process associated with Mallows permutations introduced by Gnedin and Olshanski \cite{GO10}. The regenerative structure was previously used by Basu and Bhatnagar \cite{BB17} to study the longest increasing subsequence and Gladkick and Peled to study cycles \cite{GP18}. The asymptotic correlation is given in terms of the distributions defining the Mallows process, see Proposition \ref{prop: asymp cov}.

\begin{remark}
Using Theorem 1.2 of \cite{GR96} along with the bounds obtained in Section \ref{sec: cov bounds}, an error bound can be given in terms of the difference between the correlation of $\des(w)$ and $\des(w^{-1})$ and $\rho$. Thus, getting a quantitative bound for
\begin{equation*}
    \left|\rho-\frac{\Cov(\des(w),\des(w^{-1}))}{\Var(\des(w))}\right|
\end{equation*}
would lead to a quantitative bound for Theorem \ref{thm: bivariate limit}, but this is not pursued further.
\end{remark}

\begin{remark}
The techniques introduced in \cite{BB17} to prove a central limit theorem for the longest increasing subsequence of Mallows permutations could also be used to prove Theorems \ref{thm: main theorem 1} and \ref{thm: bivariate limit} when $0<q<1$ is fixed, and indeed the computation of the asymptotic correlation in Theorem \ref{thm: bivariate limit} relies on these ideas. The benefits of the Stein's method approach of this paper are that it gives a quantitative bound, and that it allows $q$ to vary with $n$.
\end{remark}

\subsection{Related work}
The central limit theorem for $\des(w)$ is a classical result with many proofs, due to the simple dependency structure. The central limit theorem for $(\des(w),\des(w^{-1}))$ when $w$ is uniform was first studied by Vatutin \cite{V96} using generating functions. Chatterjee and Diaconis later gave a proof using Stein's method with interaction graphs \cite{CD17} and \"Ozdemir gave a proof using martingales \cite{O19}. However, none of these techniques can be applied when $w$ is Mallows distributed.

The statistic $\des(w)+\des(w^{-1})$ makes sense for any finite Coxeter group. The central limit theorem for $\des(w)+\des(w^{-1})$ for the uniform measure on finite Coxeter groups was conjectured by Kahle and Stump \cite{KS20}. It was shown to hold in a series of works \cite{R18, O19, BR19, F19}. The proof relies on the classification of finite Coxeter groups to reduce the problem to analyzing certain infinite families.

Other generalizations of descents on the symmetric group have also been considered. Conger and Viswanath study descents of permutations on multisets \cite{CV07}, using a related size-bias coupling. There has also been work on the number of descents for a permutation chosen uniformly from a conjugacy class \cite{F98,KL20}, as well as jointly with other statistics \cite{KL18, FKLP19} and for peaks \cite{FKL19}. Finally, functional central limit theorems for the process of descents has also been considered \cite{FGHHPRT20}.

The study of more general patterns within permutations has also been considered in \cite{CdS17, CdSE18} where Poisson and central limit theorems are shown for the number of occurrences of patterns within a permutation. Here, limit theorems are easier to establish due to the simple dependency structure. These results have also been generalized to multiset permutations and set partitions \cite{F20}. There is also considerable work in the combinatorial literature on pattern avoidance which is also related, see \cite{E16} for a survey with numerous references.

A growing body of work studies the behaviour of Mallows permutations as the parameter $q$ varies. In general, it seems that for $q$ sufficiently close to $1$, the model behaves very similarly to the uniform permutation, while for $q$ far enough from $1$, different behaviour occurs. One interesting feature of the Mallows model is that in many cases, there is a phase sharp phase transition between these two regimes. Specifically, this has been observed in both the cycle structure \cite{M16a, GP18} and the longest increasing subsequence \cite{MS13, BP15, BB17}.

The Mallows model also appears as the stationary distribution of the biased interchange process on the line \cite{LL19}, and it's projection to particle systems is the stationary distribution of ASEP. It also appears as the stationary distribution of random walks related to Hecke algebras \cite{DR00, B20}.

\subsection{Further directions}
The size-bias coupling constructed works very generally for Mallows models. In particular, an analogous coupling can be constructed for any local statistic of the form $\sum F_i(w)+F_i(w^{-1})$, where $F_i$ is some function of $w$ depending only on the coordinates $i, \dotsc, i+k$ for some fixed $k$. The main difficulty in general is to obtain the variance and covariance estimates needed to show the upper bound given by Stein's lemma actually goes to $0$. While this seems difficult in general, the case when $k$ is small, for example the number of peaks or valleys, or when $F$ is the indicator function for $k-1$ consecutive descents, should both be tractable.

It is also possible to define the Mallows measure on any Coxeter group (in the infinite case, $q<1$ and may need to be small). It seems likely that a central limit theorem would hold for the Mallows distribution, as in the uniform case. The Mallows measure also makes sense on infinite Coxeter groups, and it is natural to ask whether a central limit theorem should hold for nice families of infinite Coxeter groups, such as the affine symmetric groups.

Most of the preliminary results including the independence results of Section \ref{sec: mallows permutations} and the size-bias coupling of Section \ref{sec: coupling} continue to hold in general Coxeter groups with the appropriate modifications. It seems that the techniques developed in this paper should be able to answer these questions, although the estimates needed would have to be established on a case by case basis. 

It would be very interesting to see if uniform estimates could be given depending only on simple data coming from the Coxeter group structure but this is not known even in the uniform case.

\subsection{Outline}
The paper is organized as follows. Section \ref{sec: mallows permutations} explains some basic properties of Mallows permutations. Section \ref{sec: asymp cor} uses a regenerative process connected to Mallows permutations to compute the asymptotic correlation between $\des(w)$ and $\des(w^{-1})$. Section \ref{sec: coupling} reviews size-bias coupling and Stein's method, and then constructs the size-bias coupling for descents. In Section \ref{sec: uniform}, a short proof of the uniform case when $q=1$ is given. Section \ref{sec: cov bounds} gives the covariance bounds needed to control the error coming from Stein's method. Finally, Section \ref{sec: main thms} gives proofs of Theorems \ref{thm: main theorem 1} and \ref{thm: bivariate limit} along with a Poisson limit theorem.

\subsection{Notation}
Let $[n]=\{1,\dotsc,n\}$. Let $[n]_q=(1-q^n)/(1-q)$ and let $[n]_q!=\prod _{i=1}^n[i]_q$, with the convention that $[0]_q!=1$. For a set $A$, let $I_A$ denote the indicator function for that set.

\section{Properties of Mallows permutations}
\label{sec: mallows permutations}

\subsection{The Mallows distribution}
Many of the ideas in this paper used to study the Mallows distribution come from the theory of Coxeter groups. While no explicit results on Coxeter groups are needed, the notation and conventions mostly follows that of \cite{BB05}.

Fix some permutation $w\in S_n$. An \emph{inversion} of $w$ is a pair $i,j\in [n]$, with $i< j$, such that $w(i)>w(j)$. The \emph{length} of $w$, denoted $l(w)$, is the number of inversions in $w$.

Say that $w$ has a \emph{descent} at $i\in [n-1]$ if $w(i+1)<w(i)$. Let $\des_i(w)$ denote the indicator function for the event that $w$ has a descent at $i$. The \emph{descent statistic} is defined by $\des(w)=\sum \des_i(w)$ and the \emph{two-sided descent statistic} is given by $\des(w)+\des(w^{-1})$ (the name comes from the theory of Coxeter groups, where the notions of left and right descents exist).

The \emph{Mallows distribution} $\mu_q$ is a one-parameter family of probability measures on $S_n$, indexed by a real parameter $q\in (0,\infty)$. A random permutation $w\in S_n$ is Mallows distributed if
\begin{equation*}
    \Prob(w=w_0)=\frac{q^{l(w_0)}}{Z_n(q)}
\end{equation*}
where $Z_n(q)$ is a normalization constant, given explicitly as
\begin{equation*}
    Z_n(q)=[n]_q!.
\end{equation*}
When $q=1$, it reduces to the uniform distribution and when $q\to 0$ or $q\to\infty$, it degenerates to a point mass at the identity or the permutation sending $i$ to $n-i+1$ respectively. More generally, if $W=\prod S_{n_i}$ is a product of symmetric groups, the \emph{Mallows distribution} on $W$ is given by the product of independent Mallows distributions on each factor.

Let $S\subseteq [n-1]$. Say that $S$ is \emph{connected} if $i,j\in S$ implies that $k\in S$ for all $k$ such that $i<k<j$ (in other words, if it consists of consecutive numbers). The \emph{connected components} of $S$ are then the maximal connected subsets of $S$. The \emph{indices associated to} $S$, denoted $\overline{S}$, is the subset of $[n]$ defined by
\begin{equation*}
    \overline{S}=\{k\in [n]|k\in S\textnormal{ or }k-1\in S\}.
\end{equation*}

Given $S\subseteq [n-1]$, let $S_n^S$ denote the subgroup of $S_n$ generated by the elements $(i,i+1)$ for $i\in S$ (this is called a \emph{parabolic subgroup} in the Coxeter group literature). When $S$ is connected, $S_n^S\cong S_{|S|+1}$. In general, $S_n^S\cong \prod S_{n_i}$ where the $n_i$ are given by $1$ plus the sizes of the connected components of $S$.

Let $w^S$ denote the induced permutation in $S_n^S$ given by the relative order of all the numbers $w(i)$ within each connected component $S_i\subseteq S$. Note that for any $w\in S_n$, $w^S$ can be viewed to lie in a product of symmetric groups under the identification $S_n^S\cong \prod S_{n_i}$.

\begin{remark}
Note that if $S$ is not connected, then $w^S$ is not the permutation of $S_{|S|+1}$ given by the relative orders of the $w(i)$ for $i\in S$, but only the relative orders within each connected subset. For example, if $w=3251476$, and $S$=\{2,3,5,6\}, then $S_n^S\cong S_3\times S_3$ and $w^S=(231,132)$. 
\end{remark}

\subsection{Independence results}
The following results on independence of various features of Mallows permutations seem to be fairly standard, cf. \cite[Lemma 3.15]{GP18}, but the author could not find Proposition \ref{prop: ind for con set and event} stated in the literature.

\begin{proposition}
\label{prop: ind for con set and event}
Let $S\subseteq [n-1]$, and let $A\subseteq S_n$ such that for all $w\in S_n^S$, $Aw=A$. Then under the Mallows distribution, $w^S$ is Mallows distributed in $S_n^S$, and $w^S$ and $A$ are independent.
\end{proposition}
\begin{proof}
For any $w_0\in S_n^S$, note that $(ww_0)^S=w^Sw_0$. Then
\begin{equation*}
    \Prob(\{w^S=w_0\}\cap A)=\Prob((\{w^S=e\}\cap A)w_0)
\end{equation*}
because $Aw_0^{-1}=A$. Now $l(w)=l_1(w)+l_2(w)$ where if $S$ has connected components $S_i$, then $l_1(w)=\sum _i l_{S_n^{S_i}}(w^{S_i})$ denotes the sum of the lengths of each factor of $w^S$, and $l_2(w)$ is the number of remaining inversions. If $w^S=e$, then $l_1(w)=0$ and so $l(ww_0)=l_1(w_0)+l(w)$ because multiplication by $w_0$ doesn't affect $l_2(w)$ and $l_1(w)$ is a function of $w^S$, and $(ww_0)^S=w_0^S$. Then
\begin{equation*}
    \Prob((\{w^S=e\}\cap A)w_0)=q^{l(w_0)}\Prob(\{w^S=e\}\cap A).
\end{equation*}

Finally,
\begin{equation*}
    \Prob(\{w^S=w_0\}\cap A)=q^{l(w_0)}\Prob(\{w^S=e\}\cap A)
\end{equation*}
shows that $w^S$ is Mallows distributed by taking $A=S_n$, and also implies that
\begin{equation*}
    \Prob(\{w^S=w_0\}\cap A)=\Prob(w^S=w_0)\Prob(A)
\end{equation*}
so $w^S$ and $A$ are independent.
\end{proof}

Proposition \ref{prop: ind for con set and event} seems fairly useful in proving a wide variety of independence results for Mallows permutations. The next two lemmas which are needed are easy corollaries.
\begin{lemma}
\label{lem: ind for sep sets}
Let $S,S'$ be two connected subsets with $|i-j|>1$ for all $i\in S$ and $j\in S'$ and let $w$ be Mallows distributed. Then $w^S$ and $w^{S'}$ are independent Mallows permutations (with the same parameter), conditioning on any distribution of $[n]$ to the sets $\overline{S},\overline{S'}$ (and thus also without the conditioning). Moreover, $w^S$, $w^{S'}$ and $\{w(i)|i\in \overline{S'}\}$ are mutually independent.
\end{lemma}
\begin{proof}
First, note that $w^{S\cup S'}$ is Mallows distributed, and the event $A$ that $w$ sends $\overline{S}$ to some fixed subset of $[n]$ and $\overline{S'}$ to some disjoint fixed subset of $[n]$ is invariant under $S_n^{S\cup S'}$, so by Proposition \ref{prop: ind for con set and event}, $w^{S\cup S'}$ and $A$ are independent. Since $w^{S\cup S'}\in S_{|S|+1}\times S_{|S'|+1}$, with factors $w^S$ and $w^{S'}$, this implies that $w^S$ and $w^{S'}$ are independent Mallows permutations, and also independent of conditioning on the distribution of $[n]$ to $\overline{S}$ and $\overline{S'}$.

The second statement follows similarly, noting that the event $A$ assigning some fixed subset of $[n]$ to $\overline{S'}$ is also invariant under $S_n^{S\cup S'}$.
\end{proof}

\begin{lemma}
\label{lem: ind for sep sets inv}
Let $S,S'\subseteq [n-1]$ be connected subsets and let $w$ be Mallows distributed. Then conditional on 
\begin{equation*}
    \{w(i)|i\in \overline{S}\}\cap \overline{S'}
\end{equation*}
being either empty or containing one element, $w^S$ and $(w^{-1})^{S'}$ are independent and Mallows distributed.
\end{lemma}
\begin{proof}
For the case of empty intersection, note that the events
\begin{equation*}
    A=\{w\in S_n|\{w(i)|i\in \overline{S}\}\cap \overline{S'}=\emptyset\}
\end{equation*}
and
\begin{equation*}
    \{(w^{-1})^{S'}=w_0\}\cap A
\end{equation*}
are all invariant under $S_n^S$, and so by Proposition \ref{prop: ind for con set and event}
\begin{equation*}
    \Prob(w^S=w_0, (w^{-1})^{S'}=w_0'|A)=\Prob(w^S=w_0|A)\Prob((w^{-1})^{S'}=w_0'|A).
\end{equation*}
Also, $w^S$ is Mallows distributed because it is independent of $A$ (and by symmetry the same is true of $(w^{-1})^{S'}$).

In the case when the intersection has one element, the same argument works because $\{(w^{-1})^{S'}=w_0\}\cap A$ is still invariant, as $S$ is connected so if only one number from $\overline{S'}$ is used in $\overline{S}$, then $S_n^S$ cannot move that number past any other number in $\overline{S'}$.
\end{proof}

\subsection{Probability bounds}
In general, probabilities for events like $\{w|w(i)=j\}$ are hard to compute for Mallows permutations. The following lemma gives an upper bound for these types of events.
\begin{lemma}
\label{lem: prob bound}
Let $C\subseteq [n]$ and let $a_i\in [n]$ be distinct for $i\in C$. Let $w\in S_n$ be Mallows distributed with $q\leq 1$. Let $w'$ be the permutation with $w(i)=a_i$ for $i\in C$ and $w(i)<w(j)$ for $i,j\not\in C$. Then
\begin{equation*}
    \Prob(w(i)=a_i, i\in C)\leq \frac{q^{l(w')}[n-|C|]_q!}{[n]_q!}.
\end{equation*}
\end{lemma}
\begin{proof}
Let $\overline{w}$ denote the induced permutation in $S_{n-|C|}$ given by the relative order of $w(i)$ for $i\not\in C$. The map $\{w:w(i)=a_i\}\to S_{n-|C|}$ defined by $w\mapsto \overline{w}$ is bijective.

Now let $l_1(w)$ denote the inversions among indices in $C^c$ and $l_2(w)$ the remaining inversions and note that $l_1(w)=l(\overline{w})$. Note that $w'$ minimizes $l_2$ among permutations taking $i$ to $a_i$, because the number of inversions within the indices in $C$ is fixed, and for each $i\in C^c$, $w'$ minimizes the inversions between indices in $C^c$ to the left of $i$, and $i$. In addition, $l_1(w')=0$.

Then as $q\leq 1$,
\begin{equation*}
\begin{split}
    \sum _{w(i)=a_i}\frac{q^{l(w)}}{[n]_q!}&= \sum _{w(i)=a_i}\frac{q^{l_1(w)+l_2(w)}}{[n]_q!}
    \\&\leq \sum _{w(i)=a_i}\frac{q^{l(\overline{w})+l(w')}}{[n]_q!}
    \\&=\frac{q^{l(w')}[n-|C|]_q!}{[n]_q!}.
\end{split}
\end{equation*}

% The idea is that every permutation taking $i$ to $a_i$ contains at least $\min_{w'} l(w')-l_0(w')$ many descents

% First, write
% \begin{equation*}
%     \Prob(w(i)=a_i, i\in C)=\sum _{u\in S_{n-|C|}}\frac{q^{l(\overline{u'})}}{[n]_q!}
% \end{equation*}
% where $\overline{u'}\in S_n$ is defined by $\overline{u'}(i)=a_i$ and that the relative order of $\overline{u'}$ on $C^c$ agrees with $w'$.

% Then as $q<1$,
% \begin{equation*}
% \begin{split}
%     \sum _{u'\in S_{n-|C|}}\frac{q^{l(\overline{u'})}}{[n]_q!}&\leq \sum _{u'\in S_{n-|C|}}\frac{q^{\left(\min_{v'\in S_{n-|C|}} l(\overline{v'})-l(v')\right) +l(u')}}{[n]_q!}
%     \\&\leq \frac{q^{\min_{w'} l(w')-l_0(w')}[n-|C|]_q!}{[n]_q!}
% \end{split}
% \end{equation*}
% because $l_0(\overline{u'})=l(u')$ and the map $u'\mapsto \overline{u'}$ is bijective onto permutations taking $i$ to $a_i$.
\end{proof}

\subsection{Reversal symmetry}
To simplify the arguments, it will be assumed for most proofs that $q<1$. This can be done without loss of generality due to a reversal symmetry for Mallows permutations. Let $w^{rev}\in S_n$ be defined by $w^{rev}(i)=w(n-i+1)$. 
\begin{proposition}
\label{prop: desc of rev}
Let $w\in S_n$. Then $\des_i(w^{rev})=1-\des_{n-i}(w)$ and $\des_i((w^{rev})^{-1})=1-\des_i(w^{-1})$. In particular,
\begin{equation*}
\begin{split}
    \des(w)&=(n-1)-\des(w^{rev}), \\\des(w^{-1})&=(n-1)-\des((w^{rev})^{-1}).
\end{split}
\end{equation*}
\end{proposition}
\begin{proof}
Note that $w^{rev}(i)>w^{rev}(i+1)$ if and only if $w(n-i+1)>w(n-i)$ and so $w^{rev}$ has a descent at $i$ if and only if $w$ does not have a descent at $n-i$. Note that $(w^{rev})^{-1}(i)=n+1-w^{-1}(i)$. Then similarly, $(w^{rev})^{-1}(i)>(w^{rev})^{-1}(i+1)$ if and only if $w^{-1}(i)<w^{-1}(i+1)$ and so $(w^{rev})^{-1}$ has a descent at $i$ if and only if $w^{-1}$ does not have a descent at $i$.
\end{proof}

\begin{proposition}
\label{prop: distr of rev}
Let $w$ be distributed according to $\mu_q$. Then $w^{rev}$ is distributed according to $\mu_{q^{-1}}$.
\end{proposition}
\begin{proof}
Note that $l(w^{rev})={n\choose 2}-l(w)$ and the result follows.
\end{proof}

\section{Asymptotic correlation and the Mallows process}
\label{sec: asymp cor}
The goal of this section is to compute the asymptotics of the first and second moments of $\des(w)$ and $\des(w^{-1})$ and in particular show they are all of order $n$. While the first moment computations are easy, the second moment computations are non-trivial.

The main difficulty is showing that when $0<q<1$ is fixed, the asymptotic correlation
\begin{equation*}
    \rho=\lim _{n\to \infty}\frac{\Cov(\des(w),\des(w^{-1}))}{\Var(\des(w))}
\end{equation*}
exists. The proof relies on a regenerative process which can be thought of as an infinite Mallows permutation. The induced permutation given by taking a finite segment is then distributed as a Mallows permutation. These ideas were previously used in \cite{BB17} to study the longest increasing subsequence problem for Mallows permutations.

To compute the asymptotic correlation, first the covariance between $\des(w)$ and $\des(w^{-1})$ is shown to be asymptotically the same as the correlation between $\des(w^{(N)})$ and $\des((w^{(N)})^{-1})$, where $w^{(N)}$ is given by taking a deterministic number of regenerations within the Mallows process, and is thus a random permutation with a random size. The key is that the regenerative structure gives a lot of independence which allows the covariance to be computed in terms of certain distributions defining the Mallows process.

\subsection{The Mallows process}
The Mallows process was defined by Gnedin and Olshanski \cite{GO10}, who later extended the definition to a two-sided process \cite{GO12}. A more general notion of regenerative permutation on the integers was introduced in \cite{PT19}. The basic properties of the Mallows process are reviewed here, see \cite{GO10} or \cite{BB17} for proofs of these facts.

Fix $0<q<1$. The \emph{Mallows process} is a random permutation $w:\mathbf{N}\to\mathbf{N}$ defined as follows. Set $w(1)=i$ with probability $q^{i-1}(1-q)$ and given $w(1), \dotsc, w(i)$, set $w(i+1)=k\in \mathbf{N}\setminus \{w(1),\dotsc,w(i)\}$ with probability $q^{k'-1}(1-q)$ where $k'-1$ is the number of elements in $\mathbf{N}\setminus \{w(1),\dotsc,w(i)\}$ less than $k$. This is almost surely a bijection, and so $w^{-1}$ is well-defined, and is also distributed as a Mallows process.

Given any $n\in\mathbf{N}$, the relative order of $w(1),\dotsc,w(n)$ as a random element in $S_n$, denoted $w^{(n)}$, is Mallows distributed. Note that in general, it is not the case that $\left(w^{(n)}\right)^{-1}=\left(w^{-1}\right)^{(n)}$ even though they have the same distribution.

This process has a regenerative structure defined as follows. Let $T_1$ be the first time $n$ that $w(i)\leq n$ for all $i\leq n$. In other words, this is the first $n$ for which $w(1),\dotsc,w(n)$ defines a permutation of $1,\dotsc,n$. Then the process $(w(T_1+k)-T_1)_{k\geq 1}$ is equal in distribution to $(w(k))_{k\geq 1}$.

Similarly, let $T_k+\dots+T_1$ be the $k$th time that this occurs. Let $w_1$ be the permutation induced by $w(1),\dotsc,w(T_1)$ and let $w_k$ be the permutation induced by $w(T_{k-1}+1),\dotsc,w(T_k)$ for $k\geq 1$. Then it's clear that the $T_k$ and the $w_k$ are independent and identically distributed. Call the $w_k$ \emph{excursions} in the Mallows process and $T_k$ their \emph{sizes}. Moreover, the times $T_k$ are the renewal times of a renewal process.

\begin{lemma}
Let $w_0$ be distributed as an excursion in the Mallows process and $T_0$ its size. Then
\begin{equation*}
    \E(\des(w_0))=\frac{\E(T_0)q}{1+q}.
\end{equation*}
\end{lemma}
\begin{proof}
Note that
\begin{equation*}
    \lim_{n\to \infty}\frac{\des(w^{(n)})}{n-1}=\lim _{n\to \infty}\frac{\sum _{i=1}^n\des(w_i)}{\sum _{i=1}^n T_i-1}
\end{equation*}
almost surely, and by the strong law of large numbers
\begin{equation*}
    \lim_{n\to \infty}\frac{\des(w^{(n)})}{n-1}=\frac{q}{1+q}
\end{equation*}
and
\begin{equation*}
    \frac{\sum _{i=1}^n\des(w_i)}{\sum _{i=1}^n T_i-1}=\frac{\sum _{i=1}^n\des(w_i)}{n}\frac{n}{\sum _{i=1}^n T_i-1}\to\frac{\E(\des(w_0))}{\E(T_0)}.
\end{equation*}
\end{proof}

The following asymptotic moment computations for an arbitrary renewal process are well-known and can be found in \cite[pg. 386]{F71}.
\begin{lemma}
\label{lem: renewal limit computation}
Let $L_n$ denote the number of excursions in the Mallows process by time $n$ and let $T_0$ be distributed as the size of an excursion. Then as $n\to \infty$,
\begin{align*}
    \frac{\E(L_n)}{n}&\to \frac{1}{\E(T_0)},
    \\\frac{\Var(L_n)}{n}&\to \frac{\E(T_0^2)}{\E(T_0)^3}.
\end{align*}
\end{lemma}

\subsection{Markovian representation of regeneration times}
The regeneration times $T_i$ can be viewed as the return times of a Markov chain. Specifically, consider the process
\begin{equation*}
    M_n=\max_{1\leq i\leq n}w(i)-n
\end{equation*}
on $\mathbf{N}$. This is a recurrent Markov process with stationary distribution
\begin{equation*}
    \mu_j=\frac{1}{\prod_{k=1}^\infty (1-q^k)}\frac{q^j}{\prod _{k=1}^j(1-q^k)}.
\end{equation*}
The Markov process can be described in terms of the geometric random variables defining the Mallows process. Specifically, the walk can be described as moving from $M_n$ to $M_{n+1}=\max(M_n,Z_{n})-1$ where the $Z_n$ are independent geometric random variables. Let $R_i$ denote the hitting time of $i$ and let $R_i^+$ denote the return time at $i$. Then if the chain is started from $0$, $R_0^+$ is distributed as the size of an excursion in the Mallows process.

An important fact is that the regeneration times $T_i$ have finite moments. Just third moments are enough but the general proof is no harder. The proof proceeds by induction. The following extension of Lemma 4.3 from \cite{BB17} is needed.
\begin{lemma}
\label{lem: moment ineq}
For all $i\geq 1$ and all $k$, $\E_i(R_{i-1}^k)\geq \E_{i+1}(R_i^k)$ (where no claim is made as to the finiteness of the expectations).
\end{lemma}
\begin{proof}
The proof follows that of \cite{BB17}. Couple two copies of the Markov chain, one started at $i$ and one started at $i+1$, with the same underlying geometric random variables. Now if a single step is taken using some geometric random variable $Z$, then
\begin{equation*}
    \E_i(R_{i-1}^k)=\Prob(Z\leq i)+\sum _{j=i+1}^\infty \Prob(Z=j)\E_{j-1}((R_{i-1}+1)^k)
\end{equation*}
and
\begin{equation*}
    \E_{i+1}(R_i^k)=\Prob(Z\leq i)+\sum _{j=i+1}^\infty \Prob(Z=j)\E_{j-1}((R_i+1)^k).
\end{equation*}
But if $j\geq i$ and two copies of the chain are coupled and both start at $j$, then $R_i\leq R_{i-1}$ and so
\begin{equation*}
    \E_{j-1}((R_{i-1}+1)^k)\geq \E_{j-1}((R_i+1)^k)
\end{equation*}
and the lemma follows.
\end{proof}

\begin{proposition}
Let $T_0$ be the distribution of the excursion size of the Mallows process. Then $\E(T_0^k)<\infty$ for all $k\in\mathbf{N}$.
\end{proposition}
\begin{proof}
The argument follows the proof that $\E(T_0^2)<\infty$ in \cite[Lemma 4.5]{BB17}. Proceed by induction on $k$. The finiteness of $\E_0(R_0^+)$ is given by \cite[Lemma 4.5]{BB17}. Assume that $\E_0((R_0^+)^k)<\infty$.

First, it will be established that $\E_\mu(R_0^{k})<\infty$. Note that
\begin{equation*}
    \E_\mu(R_0^{k})=\E_{\mu}\left(\sum _{i=1}^j T_{i\to i-1}\right)^k=\sum_j \mu_j \sum _{l_1,\dotsc,l_j} \prod_i \E_{i}(R_{i-1}^{l_i})
\end{equation*}
because $R_0$ can be broken up into the time it takes to reach $j-1$, then $j-2$, and so on ($T_{i\to i-1}$ denotes the difference between the first time the chain hits $i-1$ and the first time it hits $i$), and these times are independent ($l_i$ denotes the number of times that the time it takes to go from $i$ to $i-1$ appears in a summand of the expansion).

Now
\begin{equation*}
    \sum_j \mu_j \sum _{l_1,\dotsc,l_j} \prod_i \E_{i}(R_{i-1}^{l_i})\leq \sum_j j^{k}\mu_j\max\left(\E_1(R_0^k),1\right)^{k}
\end{equation*}
using Lemma \ref{lem: moment ineq} and the fact that $R_0$ is integer-valued so $R_0^{l_i}\leq R_0^k$ for $l_i\leq k$. Then $\sum j^{k}\mu_j<\infty$ because $\mu_j\leq Aq^j$ for some $A>0$, and $\E_1(R_0^k)<\infty$ because $\E_0((R_0^+)^k)<\infty$ by the inductive hypothesis, and so $\E_\mu(R_0^k)<\infty$.

Then Lemma 2.23 of \cite{AF} states
\begin{equation*}
    \Prob_\mu(R_0=t-1)=\mu_0\Prob_0(R_0^+\geq t)
\end{equation*}
which gives
\begin{equation*}
    \E_0(P_{k+1}(R_0^+))=\frac{\E_\mu((R_0+1)^{k})}{\mu_0}
\end{equation*}
after multiplying by $t^{k}$ and summing over $t$, where
\begin{equation*}
    P_{k+1}(m)=\sum _{i=1}^m i^{k}
\end{equation*}
are the Faulhaber polynomials of degree $k+1$. This implies that $\E_0((R_0^+)^{k+1})<\infty$ since all other expressions in the equation are finite by the inductive hypothesis.
\end{proof}

\begin{remark}
While the Markov chain studied in \cite{BB17} is used to study the regeneration times, a related chain was studied in \cite{GP18} which is equally suitable for this purpose. Specifically, the chain at time $n$ is given by the number of $i\leq n$ for which $w(i)>n$. The return times at $0$ are equal to those of $M_n$.
\end{remark}

\subsection{Variance bounds}
To perform the comparison for $\des(w^{(n)})$ with $\des(w^{(T)})$ where $T=\sum _{i=1}^{\lceil n/\E(T_0)\rceil}T_i$, some variance bounds are needed. First, the following easy fact about renewal processes is stated. It follows from the fact that the inter-arrival times of a renewal process are asymptotically distributed as the size-bias distribution of the arrival times (see \cite[eq. 5.70]{G13} for example), and a dominated convergence argument.
\begin{lemma}
\label{lem: second moment of int cont n}
Let $T_i$ be renewal times for a renewal process with $\E(T_i^3)<\infty$. Let $X_n$ denote the size of the interval containing $n$ (that is, the random variable $T_{L_n+1}$ where $L_n$ is the number of renewals by time $n$) and let $T_0$ be distributed as the renewal time. Then
\begin{equation*}
    \E(X_n^2)\to \E((T_0^*)^2)=\frac{\E(T_0^3)}{\E(T_0)}.
\end{equation*}
\end{lemma}
Now the following variance bounds may be established.
\begin{lemma}
\label{lem: diff var bound}
Let $m=\lceil n/\E(T_0) \rceil$ and let $N=\sum _{i=1}^{m} T_i$. Then
\begin{equation*}
    \frac{\Var\left(\des(w^{(n)})-\frac{(n-1)q}{1+q}-\left(\des(w^{(N)})-\frac{Nq}{1+q}\right)\right)}{n}\to 0.
\end{equation*}
\end{lemma}
\begin{proof}
Note that $\des(w^{(N)})=\sum _{i=1}^m \des(w_i)$ and so
\begin{equation*}
    \des(w^{(n)})-\frac{(n-1)q}{1+q}-\left(\des(w^{(N)})-\frac{Nq}{1+q}\right)
\end{equation*}
is mean $0$, so it is equivalent to look at
\begin{equation*}
    \E\left(\des(w^{(n)})-\frac{(n-1)q}{1+q}-\left(\des(w^{(N)})-\frac{Nq}{1+q}\right)\right)^2.
\end{equation*}

Now let $L_n$ be the number of excursions by time $n$, and condition on $L_n<m$. Then if $L_n=l$,
\begin{equation*}
\begin{split}
    &\E\left(\left(\des(w^{(n)})-\frac{(n-1)q}{1+q}-\left(\des(w^{(N)})-\frac{Nq}{1+q}\right)\right)^2\bigg|L_n=l\right)
    \\=&\E\left(\left(\sum _{k=l+2}^m \left(\des(w_{k})-\frac{T_kq}{1+q}\right)+\sum _{k=n+1}^{\sum _{i=1}^{l+1} T_i}\left(\des_k(w)-\frac{q}{1+q}\right)\right)^2\bigg|L_n=l\right)
\end{split}
\end{equation*}
and using the fact that the $w_{k}$ and $T_k$ for $k\geq l+1$ are independent (even after conditioning on $L_n=l$ since the event is $\sigma(T_1, \dotsc, T_{l+1})$ measurable), this can be bounded by
\begin{equation}
\label{eq: variance cond on L}
    \E\left(\sum _{k=l+2}^m \des(w_{k})-\frac{T_kq}{1+q}\right)^2+\E(T_{l+1}^2|L_n=l)
\end{equation}
since $(\des_k(w)-1/(1+q))^2\leq 1$.

Now let $w_0$ be distributed as an excursion in the Mallows process and $T_0$ its size and let
\begin{equation*}
    C=\Var\left(\des(w_{0})-\frac{T_0q}{1+q}\right).
\end{equation*}
Then
\begin{equation*}
    \E\left(\sum _{k=l+2}^m \des(w_{k})-\frac{T_kq}{1+q}\right)^2\leq C|m-l|.
\end{equation*}
Multiplying \eqref{eq: variance cond on L} by $\Prob(L_n=l)$ and summing over $l$ (where each term is positive so $l\geq m$ is fine), the bound
\begin{equation}
\label{eq: L<m bound}
    \E\left(C|m-L_n|+T_{L_n+1}^2\right)
\end{equation}
is obtained.

Now $|m-\E(L_n)|$ is $o(n)$ and so 
\begin{equation*}
    \E(|m-L_n|)\leq \sqrt{\Var(L_n)}+|m-\E(L_n)|=o(n)   
\end{equation*}
by Lemma \ref{lem: renewal limit computation}. The random variable $T_{L_n+1}$ is the size of the Mallows excursion containing $n+1$, and so by Lemma \ref{lem: second moment of int cont n}, $\E(T_{L_n+1}^2)$ is bounded. Thus, \eqref{eq: L<m bound} is $o(n)$.

Note that $L_n\geq m$ is equivalent to $N\leq n$. Then
\begin{equation*}
\begin{split}
    &\E\left(\left(\des(w^{(n)})-\frac{(n-1)q}{1+q}-\left(\des(w^{(N)})-\frac{Nq}{1+q}\right)\right)^2\bigg|N\leq n\right)
    \\=&\E\left(\left(\sum _{k=N+1}^{n-1} \des_k(w)-\frac{(n-N-1)q}{1+q}\right)^2\bigg|N\leq n\right).
\end{split}
\end{equation*}
Now given $N=n_0\leq n$, if $S=\{n_0+1, \dotsc, n-1\}$ then $w^S$ is independent of the conditioning by Lemma \ref{lem: ind for sep sets} and is Mallows distributed in $S_{n-n_0}$, and as the descents $\des_k(w)$ for $k\geq n_0+1$ are a function of $w^S$. Then from the variance of $\des(w)$ for a Mallows permutation (see e.g. \cite[Proposition 5.2]{BDF10}) this is given by
\begin{equation*}
    \E(A(n-N)+B|N\leq n)\leq \frac{\E(A|n-N|+B)}{\Prob(N\leq n)}
\end{equation*}
where $A,B$ are constants with $An+b\geq 0$ for $n\geq 2$. But $|\E(N)-n|=O(1)$,
\begin{equation*}
    \E|N-\E(N)|\leq \sqrt{\Var(N)}=O(n^{\frac{1}{2}}),
\end{equation*}
and $\Prob(N\leq n)\to 1/2$ so the desired result follows.
\end{proof}

The next variance bound is needed due to the incompatibility of taking induced permutations and inversion of a Mallows process.
\begin{lemma}
\label{lem: commutator of inversion and projection}
Fix $n\in\mathbf{N}$ and let $w$ denote a Mallows process. Then
\begin{equation*}
    \Var\left(\des\left(\left(w^{(n)}\right)^{-1}\right)-\des\left(\left(w^{-1}\right)^{(n)}\right)\right)=O(1).
\end{equation*}
\end{lemma}
\begin{proof}
Although $\left(w^{(n)}\right)^{-1}\neq \left(w^{-1}\right)^{(n)}$ in general, it holds if $w(i)\leq i$ for all $i\leq n$. Let $N$ be the last time before $n$ where this occurs and let $X_n$ denote the length of the excursion containing $n$. Then
\begin{equation*}
\begin{split}
    &\Var\left(\des\left(\left(w^{(n)}\right)^{-1}\right)-\des\left(\left(w^{-1}\right)^{(n)}\right)\right)
    \\=&\Var\left(\sum _{i=N+1}^n \des_i\left(\left(w^{(n)}\right)^{-1}\right)-\des_i(w^{-1})\right)
    \\&\leq \E(X_n^2).
\end{split}
\end{equation*}
But by Lemma \ref{lem: second moment of int cont n} this converges to a finite quantity and so in particular is bounded.
\end{proof}

\subsection{Asymptotic correlation computation}
The following proposition shows that the limiting correlation $\rho$ exists if $0<q<1$ is fixed.
\begin{proposition}
\label{prop: asymp cov}
Fix $0<q<1$, let $v\in S_n$ be Mallows distributed, and let $w_0$ be distributed as an excursion in the Mallows process with $T_0$ its size. Then
\begin{equation*}
    \frac{\Cov(\des(v),\des(v^{-1}))}{\Var(\des(w))}\to \frac{\Cov\left(\des(w_0)-\frac{T_0q}{1+q},\des(w_0^{-1})-\frac{T_0q}{1+q}\right)}{\Var\left(\des(w_0)-\frac{T_0q}{1+q}\right)}.
\end{equation*}
\end{proposition}
\begin{proof}
Let $w$ be a Mallows process with excursion sizes $T_i$, and couple it to $v$ so $w^{(n)}=v$. If $m=\lceil n/\E(T_0)\rceil$ and $N=\sum _{i=1}^m T_i$,
\begin{equation*}
    \frac{\Cov\left(\des\left(w^{(N)}\right)-\frac{Nq}{1+q},\des\left(\left(w^{(N)}\right)^{-1}\right)-\frac{Nq}{1+q}\right)}{n}
\end{equation*}
converges to
\begin{equation*}
    \frac{1}{\E(T_0)}\Cov\left(\des(w_0)-\frac{T_0q}{1+q},\des(w_0^{-1})-\frac{T_0q}{1+q}\right)
\end{equation*}
by independence of the $w_k$ and $T_k$. Now
\begin{equation*}
    \left|\Cov(\des(v),\des(v^{-1}))-\Cov\left(\des\left(w^{(N)}\right)-\frac{Nq}{1+q},\des\left(\left(w^{(N)}\right)^{-1}\right)-\frac{Nq}{1+q}\right)\right|
\end{equation*}
is bounded by
\begin{equation*}
\begin{split}
    &\E\bigg(\left|\des\left(w^{(n)}\right)-\frac{(n-1)q}{1+q}\right|\left|\des\left(\left(w^{-1}\right)^{(n)}\right)-\des\left(\left(w^{(n)}\right)^{-1}\right)\right|
    \\&\quad +\left|\des\left(w^{(n)}\right)-\frac{(n-1)q}{1+q}\right|\left|\des\left(\left(w^{-1}\right)^{(n)}\right)-\des\left(\left(w^{-1}\right)^{(N)}\right)-\frac{(n-1-N)q}{1+q}\right|
    \\&\quad +\left|\des\left(\left(w^{(N)}\right)^{-1}\right)-\frac{Nq}{1+q}\right|\left|\des\left(w^{(n)}\right)-\des\left(w^{(N)}\right)-\frac{(n-1-N)q}{1+q}\right|\bigg).
\end{split}
\end{equation*}
Now apply Cauchy-Schwarz and note that $\Var(\des(w))=O(n)$ and $\Var(\des(w^{(N)})-Nq/(1+q))=O(n)$, and then Lemma \ref{lem: commutator of inversion and projection} gives an $o(n)$ bound for the first term and Lemma \ref{lem: diff var bound} gives an $o(n)$ bound for the other two terms. Thus,
\begin{equation*}
    \frac{\Cov(\des(w),\des(w^{-1}))}{n}\to \frac{1}{\E(T_0)}\Cov\left(\des(w_0)-\frac{T_0q}{1+q},\des(w_0^{-1})-\frac{T_0q}{1+q}\right).
\end{equation*}

Similarly,
\begin{equation*}
    \frac{\Var(\des(w))}{n}\to \frac{1}{\E(T_0)}\Var\left(\des(w_0)-\frac{T_0q}{1+q}\right)
\end{equation*}
and the desired result follows.
\end{proof}

\begin{remark}
\label{rmk: asymp cov}
While Proposition \ref{prop: asymp cov} gives the existence of the asymptotic correlation, it does not give an explicit formula. In particular, it seems hard to understand the behaviour of excursions in the Mallows process, and even sampling such excursions becomes difficult when $q$ is close to $1$.

See Figure \ref{fig:1} for some simulations of the relationship between $\rho$ and $q$. Note that as $q\to 1$, the expected size of Mallows excursions goes to infinite so excursions are only sampled up to $q=0.8$.

One non-trivial consequence is that because $\des(w_0)$ and $\des(w_0^{-1})$ are not equal, the asymptotic correlation between $\des(w)$ and $\des(w^{-1})$ is strictly less than $1$.
\end{remark}
\begin{figure}
    \centering
    \includegraphics{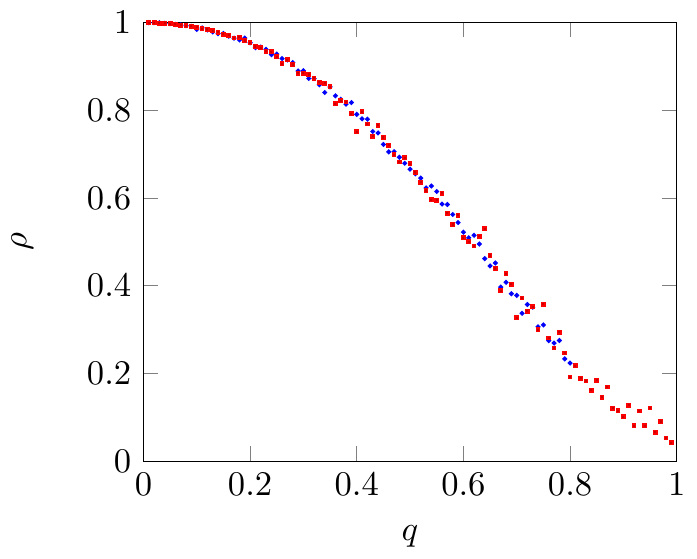}
    \caption{Sample correlation between $\des(w)$ and $\des(w^{-1})$, $w\in S_{1000}$ Mallows distributed, shown in red (100 data points each estimated with 1000 samples). Sample correlation between $\des(w_0)-T_0q/(1+q)$ and $\des(w_0^{-1})-T_0q/(1+q)$, $w_0$ a Mallows excursion and $T_0$ its size, shown in blue (80 data points, $q$ up to $0.8$, each estimated with 10000 samples).}
    \label{fig:1}
\end{figure}

\subsection{Moment computations}
The following moment bounds are relatively straight-forward. Note that although Proposition \ref{prop: asymp cov} gives the asymptotic correlation, for the proof of Theorem \ref{thm: main theorem 1} quantitative bounds are needed for finite $n$.

\begin{proposition}
\label{prop: mean var calc}
Let $w\in S_n$ be Mallows distributed with parameter $q<1$ and $n\geq 2$. Then
\begin{equation*}
        \E(\des(w)+\des(w^{-1}))=\frac{2q(n-1)}{1+q},
\end{equation*}
\begin{align*}
       \Cov(\des(w),\des(w^{-1}))&\geq \frac{q(n-1)(1-q)^2\prod _{k=1}^\infty (1-q^k)}{(1+q)},
       \\\Cov(\des(w),\des(w^{-1}))&\leq \frac{q(n-1)(1-q)(1+q)^2}{(1-q^n)}
\end{align*}
and
\begin{equation*}
    \Var(\des(w)+\des(w^{-1}))\geq\frac{2nq(1-q+q^2)}{(1+q)^2(1+q+q^2)}-\frac{2q(1-3q+q^2)}{(1+q)^2(1+q+q^2)}.
\end{equation*}
\end{proposition}
\begin{proof}
For the expectation, note that $\des(w)+\des(w^{-1})$ is a sum of $2(n-1)$ identically distributed random variables $\des_i(w)$ or $\des_i(w^{-1})$, each of which has mean $q/(1+q)$.

Now, consider $\Cov(\des_i(w),\des_j(w^{-1}))$ for some $i,j\in [n-1]$. First, condition on the size of $\{w(i),w(i+1)\}\cap \{j,j+1\}$. If the intersection is empty or has one element, then $\des_i(w)$ and $\des_j(w^{-1})$ are independent and both Bernoulli with parameter $q/(1+q)$ by Lemma \ref{lem: ind for sep sets inv}.

Finally, suppose that the intersection has two elements. Then of course $\des_i(w)=\des_j(w^{-1})$, and moreover $\des_i(w)$ is Bernoulli with parameter $q/(1+q)$ (because by Lemma \ref{lem: ind for sep sets} it's independent of the conditioning). Then
\begin{equation}
\label{eq: cov of des at i and ides at j}
    \Cov(\des_i(w),\des_j(w^{-1}))=\left(\frac{q}{1+q}-\frac{q^2}{(1+q)^2}\right)\Prob(\{w(i),w(i+1)\}=\{j,j+1\}).
\end{equation}

Now the upper bound can be obtained by applying Lemma \ref{lem: prob bound}, giving
\begin{equation*}
\begin{split}
    \Prob(\{w(i),w(i+1)\}=\{j,j+1\})&=(1+q)\Prob(w(i)=j, w(i+1)=j+1)
    \\&\leq \frac{(1+q)q^{2|i-j|}(1-q)^2}{(1-q^n)(1-q^{n-1})}
\end{split}
\end{equation*}
and summing \eqref{eq: cov of des at i and ides at j} over $i,j$ gives the bound
\begin{equation*}
    \Cov(\des(w),\des(w^{-1}))\leq\frac{q(n-1)(1-q)(1+q)^2}{(1-q^{n})}.
\end{equation*}

The lower bound is giving by summing over $i=j$. Note
\begin{equation*}
\begin{split}
    \Prob(\{w(i),w(i+1)\}=\{i,i+1\})&\geq(1+q) \frac{[i-1]_q![n-i-1]_q!}{[n]_q!}
    \\&\geq(1-q)\dotsm (1-q^{i-1})(1-q)^2(1+q)
    \\&\geq (1+q)(1-q)^2\prod _{k=1}^\infty (1-q^k)
\end{split}
\end{equation*}
independent of $i$, and summing \eqref{eq: cov of des at i and ides at j} over $i=j$ gives the desired lower bound.

For the variance, first note that
\begin{equation*}
    \Var(\des(w))=\frac{nq(1-q+q^2)}{(1+q)^2(1+q+q^2)}-\frac{q(1-3q+q^2)}{(1+q)^2(1+q+q^2)},
\end{equation*}
see \cite[Proposition 5.2]{BDF10}, and since the covariance is non-negative the lower bound follows.
\end{proof}

In particular, the bounds on the covariance imply that if $q\to 1$, then $\des(w)$ and $\des(w^{-1})$ are asymptotically uncorrelated and if $q\to 0$, then $\des(w)$ and $\des(w^{-1})$ are asymptotically perfectly correlated.
\begin{corollary}
\label{cor: asymp cor 1 0}
If $q_n\to 1$ and $w$ is Mallows distributed with parameter $q_n$, then
\begin{equation*}
    \frac{\Cov(\des(w),\des(w^{-1}))}{\Var(\des(w))}\to 0
\end{equation*}
and if $q_n\to 0$ (or $q_n\to\infty$), then
\begin{equation*}
    \frac{\Cov(\des(w),\des(w^{-1}))}{\Var(\des(w))}\to 1.
\end{equation*}
\end{corollary}
\begin{proof}
It suffices to check that the corresponding bounds given by Proposition \ref{prop: mean var calc} go to either $0$ or $1$ respectively. 

The covariance is positive, so for the upper bound it suffices to check that
\begin{equation*}
    \frac{1-q_n}{1-q_n^n}\to 0.
\end{equation*}
For $q_n\leq 1-1/n$,
\begin{equation*}
    \frac{1-q_n}{1-q_n^n}\leq \frac{1-q_n}{1-e^{-1}}
\end{equation*}
and for $q_n\geq 1-1/n$,
\begin{equation*}
    \frac{1-q_n}{1-q_n^n}\leq \frac{4}{n}
\end{equation*}
when $n\geq 2$.

The lower bound converges to $1$ and the case when $q_n\to \infty$ can be handled by symmetry.
\end{proof}

\section{Size-bias coupling for two-sided descents}
\label{sec: coupling}
In this section, size-bias couplings and their relation to Stein's method are reviewed. Then a size-bias coupling for the two-sided descent statistic is constructed, based on a coupling due to Goldstein \cite{G05} which was also used by Conger and Viswanath \cite{CV07} to study descents in multiset permutations.

\subsection{Size-bias coupling and Stein's method}
Let $X$ be a non-negative discrete random variable with positive mean. Say that $X^*$ has the \emph{size-bias distribution} with respect to $X$ if
\begin{equation*}
    \Prob(X^*=x)=\frac{x\Prob(X=x)}{\E(X)}.
\end{equation*}
A \emph{size-bias coupling} is a pair $(X,X^*)$ of random variables, defined on the same probability space such that $X^*$ has the size-bias distribution of $X$.

The following version of Stein's method using size-bias coupling is the main tool used to prove Theorem \ref{thm: main theorem 1}.

\begin{theorem}[{\cite[Theorem 1.1]{GR96}}]
\label{thm: size-bias steins}
Let $X$ be a non-negative random variable, with $\E(X)=\mu$ and $\Var(X)=\sigma^2$ and let $(X,X^*)$ be a size-bias coupling. Let $Z$ denote a standard normal random variable. Then for all piecewise continuously differentiable functions $h:\mathbf{R}\to \mathbf{R}$,
\begin{equation*}
    \left|\E h\left(\frac{X-\mu}{\sigma}\right)-\E h(Z)\right|\leq 2\|h\|_\infty \frac{\mu}{\sigma^2}\sqrt{\Var \E(X-X^*|X)}+\|h'\|_\infty \frac{\mu}{\sigma^3}\E(X-X^*)^2.
\end{equation*}
\end{theorem}
This theorem implies a central limit theorem if both error terms can be controlled. For the application to two-sided descents, the relevant error term will be the first one, because $|X-X^*|$ will be bounded, and $\mu,\sigma^2$ are both of order $n$.

The following construction of a size-bias coupling for sums of random variables is also crucial.

\begin{lemma}[{\cite[Lemma 2.1]{BRS89}}]
\label{lem: size-bias characterization}
Let $X=\sum X_i$ be a sum of non-negative random variables. Let $I$ be a random index with the distribution $\Prob(I=i)=\E X_i/\sum \E X_j$. Let $X^*=\sum X_i'$ where conditional on $I=i$, $X_i'$ has the size-bias distribution of $X_i$ and
\begin{equation*}
    \Prob((X_1',\dotsc,X_n')\in A|I=i,X_i'=x)=\Prob((X_1,\dotsc,X_n)\in A|X_i=x).
\end{equation*}
Then $X^*$ has the size-bias distribution of $X$.
\end{lemma}

\subsection{Construction}
Let $w\in S_n$ be a permutation. For $i\in [n-1]$, let $w_i^*$ be the permutation with $w_i^*(j)=w(j)$ for $j\neq i,i+1$ and $w_i^*(i)>w_i^*(i+1)$. Call this process \emph{reverse sorting} $w$ at $i$. Write $((w^{-1})_i^*)^{-1}=w_{-i}^*$. This corresponds to reverse sorting not the numbers at the locations $i,i+1$ but the numbers $i,i+1$ themselves.

Consider the random permutation $w^*$ obtained from $w$ by taking a random integer $i\in [n-1]$ and a sign $\pm$ uniformly at random, and taking $w^*=w^*_{\pm i}$. The claim is that this gives a size-bias coupling for two-sided descents.

\begin{proposition}
\label{prop: size-bias coupling}
Let $w\in S_n$ be Mallows distributed. Then the random variable $\des(w^*)+\des((w^*)^{-1})$ has the size-bias distribution of $\des(w)+\des(w^{-1})$.
\end{proposition}
\begin{proof}
The proof proceeds by showing that it coincides with the construction given by Proposition \ref{lem: size-bias characterization}.

Write
\begin{equation*}
    \des(w)+\des(w^{-1})=\sum _{i=1}^{n-1}\des_i(w)+\des_i(w^{-1}).
\end{equation*}
Note that each $\des_i(w)$ and $\des_i(w^{-1})$ is identically distributed (although they are not independent), with $\Prob(\des_i(w)=1)=q/(1+q)$. The size-bias distribution of $\des_i(w)$ is the constant $1$. Then $\des(w^*)+\des((w^*)^{-1})$ can be described as picking a summand uniformly at random, and replacing it with its size-bias distribution.

Thus, it remains to check that conditional on picking some index $i$, and either $w$ or $w^{-1}$, which by symmetry can be taken to be $w$, the distribution of all other summands $\des_j(w_i^*)$, $j\neq i$, and $\des_j((w_i^*)^{-1})$ is the same as that of $\des_j(w)$, $j\neq i$ and $\des_j(w^{-1})$ conditioned to have $\des_i(w)=1$ (note that the conditioning gives $w^*=w_i^*$).

In fact, it can be shown that the distribution of $w_i^*$ is equal to that of $w$ conditioned to have $\des_i(w)=1$. To see this, note that
\begin{equation*}
\begin{split}
    \Prob(w_i^*=w_0)&=\Prob(w=w_0|\des_i(w)=1)\Prob(\des_i(w)=1)
    \\&\qquad+\Prob(w_i^*=w_0|\des_i(w)=0)\Prob(\des_i(w)=0)
\end{split}
\end{equation*}
and so it suffices to show that the distribution of $w_i^*$ given that $\des_i(w)=0$ and the distribution of $w$ given that $\des_i(w)=1$ are the same. Note that on the event that $\des_i(w)=0$, $w_i^*=w(i,i+1)$. But the map $w\mapsto w(i,i+1)$ is a bijection from $\{w|\des_i(w)=0\}$ to $\{w|\des_i(w)=1\}$ and $\Prob(w(i,i+1))=q\Prob(w)$ if $\des_i(w)=0$. Thus, the relative probabilities are unchanged and so the distributions are the same.
\end{proof}

While the main goal is to prove a central limit theorem, size-bias couplings are quite powerful, and in particular the fact that the above coupling satisfies
\begin{equation*}
    \des(w^*)+\des((w^*)^{-1})\leq \des(w)+\des(w^{-1})+2
\end{equation*}
immediately implies the following tail bounds, see \cite[Theorem 3.3]{CGJ18}.
\begin{proposition}
\label{prop: tail bounds}
Let $w\in S_n$ be Mallows distributed and let $\mu=2(n-1)q/(1+q)$. Then
\begin{equation*}
    \Prob\left(\des(w)+\des(w^{-1})-\mu\geq x\right)\leq \exp{\left(-\frac{x^2}{4(x/3+\mu)}\right)}
\end{equation*}
for $x\geq 0$ and
\begin{equation*}
    \Prob\left(\des(w)+\des(w^{-1})-\mu\leq -x\right)\leq \exp{\left(-\frac{x^2}{4\mu}\right)}
\end{equation*}
for $0\leq x< \mu$.
\end{proposition}

\subsection{Decomposition of the variance term}
The idea is to now use this coupling to apply Theorem \ref{thm: size-bias steins} to obtain a quantitative bound. The main focus will be on the first error term, which will be called the \emph{variance term}.

First, note that 
\begin{equation*}
\begin{split}
&\Var \E(\des(w)+\des(w^{-1})-\des(w^*)-\des((w^*)^{-1})|\des(w^{-1})+\des(w))
\\\leq &\Var(\E(\des(w)+\des(w^{-1})-\des(w^*)-\des((w^*)^{-1})|w)).
\end{split}
\end{equation*}
Writing
\begin{equation*}
    E(\des(w^*)|w)=\frac{1}{2(n-1)}\sum _{i,\pm}\des(w_{\pm i}^*)
\end{equation*}
and similarly for $(w^*)^{-1}$, the variance can be written as
\begin{equation*}
    \Var\left(\frac{1}{2n}\left(\Sigma_1+\Sigma_2+\Sigma_3+\Sigma_4\right) \right)
\end{equation*}
where
\begin{align*}
    \Sigma_1&=\sum _i \des(w)-\des(w_i^*)
    \\\Sigma_2&=\sum _i \des(w)-\des(w_{-i}^*)
    \\\Sigma_3&=\sum _i \des(w^{-1})-\des((w_i^*)^{-1})
    \\\Sigma_4&=\sum_i \des(w^{-1})-\des((w_{-i}^*)^{-1}).
\end{align*}

The variance can be split into 16 types of covariance terms. These can be reduced by symmetry and the fact that $w$ and $w^{-1}$ have the same distribution to 6 types of terms. For example,
\begin{equation*}
\begin{split}
    &\Cov(\des(w)-\des(w_i^*),\des(w)-\des(w_j^*))
    \\=&\Cov(\des(w^{-1})-\des((w_{-i}^*)^{-1}),\des(w^{-1})-\des((w_{-j}^*)^{-1})).
\end{split}
\end{equation*}
The types of terms are summarized in Table \ref{tab:term types} along with their multiplicities.

\begin{table}
\caption{The types of terms and multiplicities in the variance bound}
\begin{tabular}{c|c|c}
     Type & Multiplicity & Term \\\hline
     1&2& $\sum \Cov(\des(w)-\des(w_i^*),\des(w)-\des(w_j^*))$\\
     2&4& $\sum \Cov(\des(w)-\des(w_i^*),\des(w^{-1})-\des((w_j^*)^{-1}))$\\
     3&4& $\sum \Cov(\des(w)-\des(w_i^*),\des(w)-\des(w_{-j}^*))$\\
     4&2& $\sum \Cov(\des(w)-\des(w_i^*),\des(w^{-1})-\des((w_{-j}^*)^{-1}))$\\
     5&2& $\sum \Cov(\des(w)-\des(w_{-i}^*),\des(w)-\des(w_{-j}^*))$\\
     6&2& $\sum \Cov(\des(w)-\des(w_{-j}^*),\des(w^{-1})-\des((w_i^*)^{-1}))$
\end{tabular}
\label{tab:term types}
\end{table}

\section{The uniform case}
\label{sec: uniform}
In this section, a new proof of Theorem \ref{thm: main theorem 1} is given in the uniform case when $q=1$. While this follows from Theorem \ref{thm: main theorem 1}, the constant can be sharpened considerably and the proof simplifies greatly and so is included.

The following lemma is straightforward.
\begin{lemma}
\label{lem: cond ind bound}
Let $X$ and $Y$ be random variables with $|X|\leq C_1$ and $|Y|\leq C_2$. Let $A$ be some event such that conditional on $A$, $X$ and $Y$ are uncorrelated. Then
\begin{equation*}
\begin{split}
    |\Cov(X,Y)|\leq& 4C_1C_2\Prob(A^c).
\end{split}
\end{equation*}
\end{lemma}

The following moment computations are also needed, see for example \cite{CD17}.
\begin{proposition}
Let $w\in S_n$ be a uniform random permutation with $n\geq 2$. Then
\begin{align*}
    \E(\des(w)+\des(w^{-1}))&=n-1,
    \\\Var(\des(w)+\des(w^{-1}))&=\frac{n+7}{6}-\frac{1}{n}.
\end{align*}
\end{proposition}

\begin{theorem}
Let $w\in S_n$ be a uniform random permutation. Then
\begin{equation*}
    \left|\E h\left(\frac{\des(w)+\des(w^{-1})-\mu}{\sigma}\right)-\E h(Z)\right|\leq \left(24\sqrt{37}\|h\|_\infty +24\sqrt{6}\|h'\|_\infty\right) (n-1)^{-\frac{1}{2}}.
\end{equation*}
with $\mu=\E(\des(w)+\des(w^{-1}))$ and $\sigma^2=\Var(\des(w)+\des(w^{-1}))$.
\end{theorem}
\begin{proof}
The second term is easy to control since $|X-X^*|\leq 2$. To see this, note that reverse sorting the numbers at $i,i+1$ can only introduce at most one descent (at $i$) and the effect on the inverse is to reverse sort the numbers $i,i+1$, which also can add at most one descent. Conversely, it can remove at most two descents from $w$ (at $i-1$ and $i+1$), and cannot remove any descents from $w^{-1}$, since $j>i$ if and only if $j>i+1$ and $j<i$ if and only if $j<i+1$ for $j\neq i,i+1$.

First, consider the terms of types 2, 3, 5, 6. These all contain either $\des(w)-\des(w_{-i}^*)$ or $\des(w^{-1})-\des((w_i^*)^{-1})$, which are equal in distribution. Without loss of generality, consider $\des(w)-\des(w_{-i}^*)$. Note that this random variable is either $0$ or $-1$, and it is $-1$ if and only if $i,i+1$ are adjacent and $i$ appears before $i+1$ in the permutation. But this happens with probability bounded by $n^{-1}$, and the other argument in the covariance is bounded by $2$, and so
\begin{equation*}
    |\Cov(\des(w^{-1})-\des((w_i^*)^{-1}),X)|\leq 2n^{-1}
\end{equation*}
by the Cauchy-Schwarz inequality. Together, there are $12(n-1)^2$ such terms, and so this contributes $6n^{-1}$ to the variance.

Next, consider terms of type 1. Note that $\des(w)-\des(w_i^*)$ and $\des(w)-\des(w_j^*)$ are independent if $|i-j|>3$. To see this, note that $\des(w)-\des(w_i^*)$ depends only on the relative order of $w(i-1)$, $w(i)$, $w(i+1)$ and $w(i+2)$. But the relative orders of disjoint subsets of $w(i)$ for $i\in S$ and $w(j)$ for $j\in S'$ are independent if $S,S'$ are disjoint. The condition $|i-j|>3$ ensures this is the case. Since these are also bounded by $2$, each term contributes at most $4$. There are at most $14(n-1)$ such terms, this contributes $14(n-1)^{-1}$ to the variance.

Finally, consider terms of type 4. This time, $\des(w)-\des(w_i^*)$ and $\des(w^{-1})-\des((w_{-j}^*)^{-1})$ are independent conditional on the sets $\{w(i-1),w(i),w(i+1),w(i+2)\}$ and $\{j-1,j,j+1,j+2\}$ being disjoint (call this event $A$). To see this, note that for any set $\{a,b,c,d\}$ disjoint from $\{j-1,j,j+1,j+2\}$, conditioning on $\{w(i-1),w(i),w(i+1),w(i+2)\}=\{a,b,c,d\}$, the joint distribution of the relative order of $w(i-1)$, $w(i)$, $w(i+1)$ and of $w(i+2)$ $j-1$, $j$, $j+1$ and $j+2$ is independent uniform. Thus, the same is true on the union of these events over subsets disjoint from $\{j-1,j,j+1,j+2\}$.

Now $\Prob(A)\geq 1-16(n-1)^{-1}$ since the complement is contained the the union of the events $\{w(k)=l\}$ for $k\in \{i-1,i,i+1,i+2\}$ and $l\in \{j-1,j,j+1,j+2\}$. Then
\begin{equation*}
    \Cov(\des(w)-\des(w_i^*),\des(w^{-1})-\des((w_{-j}^*)^{-1}))\leq 16P(A^c)
\end{equation*}
by Lemma \ref{lem: cond ind bound} since both arguments for the covariance are bounded by $2$. This gives a bound of $128(n-1)^{-1}$ on the contribution from these terms.

Combining these computations, the variance is bounded by $148(n-1)^{-1}$ which gives the claimed bound.
\end{proof}

\section{Covariance bounds}
\label{sec: cov bounds}
The structure of the proof of Theorem \ref{thm: main theorem 1} is similar to the one given for the uniform case but the bounds are more involved. This section is devoted to obtaining the necessary bounds on the terms of types 1 up to 6 as defined in Table \ref{tab:term types}. While terms of types 1 up to 4 are relatively straightforward, terms of types 5 and 6 require some delicate analysis. 

For the rest of this section, $w\in S_n$ is Mallows distributed.

\subsection{Covariance bounds: types 1, 2, 3, and 4}
\begin{lemma}
\label{lem: type 1}
The type 1 terms are bounded by $56(n-1)$.
\end{lemma}
\begin{proof}
The proof is the same as the uniform case. Note $\des(w)-\des(w_i^*)$ depends only on $w^{\{i-1,i,i+1\}}$. Thus, if $|i-j|>3$, then by Lemma \ref{lem: ind for sep sets}, $\des(w)-\des(w_i^*)$ and $\des(w)-\des(w_j^*)$ are independent. As they are bounded by $2$, each term contributes at most $4$ and there are $14(n-1)$ terms, giving the $56(n-1)$ bound.
\end{proof}

\begin{lemma}
\label{lem: type 2}
The type 2 terms are bounded by $56(n-1)$.
\end{lemma}
\begin{proof}
Note that $\des(w^{-1})-\des((w_j^*)^{-1})$ is equal to $-1$ if and only if $w(j)>w(j+1)$ and $|w(j)-w(j+1)|=1$ (because reverse sorting at position $j$ only matters if $w$ is sorted at $j$, and swapping $j,j+1$ in $w^{-1}$ only affects descents if $j,j+1$ are adjacent in $w^{-1}$) and is $0$ otherwise. If $|i-j|>3$, then $\des(w)-\des(w_i^*)$ depends only on $w^{\{i-1,i,i+1\}}$ and $\des(w^{-1})-\des((w_j^*)^{-1})$ depends only on $w^{\{j\}}$ and the set $\{w(j),w(j+1)\}$. Then by Lemma \ref{lem: ind for sep sets}, they are independent if $|i-j|>3$.

Then similar to the type 1 case, there are at most $14(n-1)$ terms which contribute, each bounded by $4$ and this gives the desired bound.
\end{proof}

\begin{lemma}
\label{lem: type 3}
The type 3 terms are bounded by $128(n-1)$.
\end{lemma}
\begin{proof}
Consider a term $\Cov(\des(w)-\des(w_i^*),\des(w)-\des(w^*_{-j}))$. Note that after conditioning on the event
\begin{equation*}
A=\{w|\{w(i-1),w(i),w(i+1),w(i+2)\}\cap \{j,j+1\}=\emptyset\},
\end{equation*}
the covariance vanishes by Lemma \ref{lem: ind for sep sets inv}, since $\des(w)-\des(w_i^*)$ is a function of $w^{\{i-1,i,i+1\}}$ and $\des(w)-\des(w^*_{-j})$ is a function of $(w^{-1})^{\{j\}}$.

By Lemma \ref{lem: cond ind bound},
\begin{equation*}
    |\Cov(\des(w)-\des(w_i^*),\des(w^{-1})-\des((w^*_j)^{-1}))|\leq 16\Prob(A^c).
\end{equation*}
Now $A^c$ is contained in the union of the sets $\{w|w(k)=l\}$ for $k\in \{i-1,i,i+1,i+2\}$ and $l\in\{j,j+1\}$. Thus, the total contribution after summing over $i,j$ is
\begin{equation*}
    16\sum _{i,j}\sum _{k,l}\Prob(w(k)=l)\leq 128(n-1)
\end{equation*}
and this gives the desired bound.
\end{proof}

\begin{lemma}
\label{lem: type 4}
The type 4 terms are bounded by $256(n-1)$.
\end{lemma}
\begin{proof}
Similar to the previous case, $\des(w)-\des(w_i^*)$ depends only on $w^{\{i-1,i,i+1\}}$ and $\des(w^{-1})-\des((w^*_{-j})^{-1})$ depends only on $(w^{-1})^{\{j-1,j,j+1\}}$. Thus, by Lemma \ref{lem: ind for sep sets inv}, after conditioning on $\{w(i-1),w(i),w(i+1),w(i+2)\}\cap \{j-1,j,j+1,j+2\}=\emptyset$ (call this event $A$), the two are independent. Then Lemma \ref{lem: cond ind bound} gives a bound of $16\Prob(A^c)$.

Now note that $A^c$ is contained in the union of the events $w(k)=l$ for $k\in \{i-1,i,i+1,i+2\}$ and $l\in \{j-1,j,j+1,j+2\}$ and so after summing over $i,j$, the bound
\begin{equation*}
    16\sum _{i,j}\sum _{k,l}\Prob(w(k)=l)\leq 256(n-1)
\end{equation*}
is obtained.
\end{proof}

\subsection{Covariance bounds: types 5 and 6}
The terms of types 5 and 6 are similar, so focus on terms of type 5. Then a bound is needed for
\begin{equation*}
    \sum _{i,j}\Cov(\des(w)-\des(w_{-i}^*),\des(w)-\des(w_{-j}^*)).
\end{equation*}
Now Let $A_i$ denote the event that $w(i+1)-w(i)=1$. Then $\des(w)-\des(w_{-i}^*)=-I_{A_i}$, and so the covariance is equal to
\begin{equation*}
    \sum _{i,j}\Prob\left(\substack{w(i+1)-w(i)=1\\w(j+1)-w(j)=1}\right)-\Prob(w(i+1)-w(i)=1)\Prob(w(i+1)-w(i)=1).
\end{equation*}
In almost all cases, it will suffice to just bound
\begin{equation*}
    \Prob\left(\substack{w(i+1)-w(i)=1\\w(j+1)-w(j)=1}\right)
\end{equation*}
but in Lemma \ref{lem: type 5 off-diag} more care is needed to take advantage of the cancellations that occur.

The strategy to obtain the necessary bounds will be to consider two regimes, one where $q\approx 1$ and one where $q\ll 1$. In the second regime, separate bounds for the terms close to the diagonal $i=j$ and those far from the diagonal are needed.

\subsubsection{Bounds for \texorpdfstring{$q\approx 1$}{ q approximately 1}}

\begin{lemma}
\label{lem: prob of adj bound}
Let $|i-j|>1$ and let $n\geq 4$. Then
\begin{equation*}
    \Prob\left(\substack{w(i+1)-w(i)=1\\w(j+1)-w(j)=1}\right)\leq \frac{12(1-q)^2}{(1-q^{n-3})^2}.
\end{equation*}
\end{lemma}
\begin{proof}
Note that
\begin{equation}
\Prob\left(\substack{w(i+1)-w(i)=1\\w(j+1)-w(j)=1}\right)\leq \sum _{k,l}\Prob\left(\substack{w(i)=k,w(i+1)=k+1\\w(j)=l,w(j+1)=l+1}\right)
\end{equation}
and by Lemma \ref{lem: prob bound},
\begin{equation}
    \Prob\left(\substack{w(i)=k,w(i+1)=k+1\\w(j)=l,w(j+1)=l+1}\right)\leq \frac{q^{\max(2|i-k|+2|j-l|-4,0)}(1-q)^4}{(1-q^n)(1-q^{n-1})(1-q^{n-2})(1-q^{n-3})}
\end{equation}
since
\begin{equation*}
    l(w')\geq \max(2|i-k|+2|j-l|-4,0)
\end{equation*}
because if $i<k$, then $w(i)=k$ forces at least $k-i$ inversions by the pigeonhole principle, and similarly if $i>k$, and so the total number of inversions forced by setting $w(i)=k$, $w(i+1)=k+1$, $w(j)=l$ and $w(j+1)=l+1$ is $2|i-k|+2|j-l|-4$, where the $-4$ comes from the fact that $k$, $k+1$, $l$ and $l+1$ might be double-counted and the quantity is at least $0$.

Summing over $k$ and $l$ gives
\begin{equation*}
    \sum _{k,l}q^{\max(2|i-k|+2|j-l|-4,0)}\leq 4\frac{(1-q^{n-1})^2}{(1-q)^2}+8
\end{equation*}
and the result follows.
\end{proof}

\begin{lemma}
\label{lem: type 5 large}
Fix $q\geq 1-(n-1)^{-1/2}$ with $n\geq 4$. Then
\begin{equation}
\label{eq: sum large}
    \sum _{i,j}\Prob\left(\substack{w(i+1)-w(i)=1\\w(j+1)-w(j)=1}\right)\leq 111(n-1).
\end{equation}
\end{lemma}
\begin{proof}
First, sum over the terms with $|i-j|\leq 1$ giving $3(n-1)$. Then using Lemma \ref{lem: prob of adj bound} to bound the terms in the sum \eqref{eq: sum large} and then summing over $i,j$ gives the desired result, noting that
\begin{equation*}
    \frac{1-q}{1-q^{n-3}}\leq \frac{(n-1)^{-1/2}}{1-(1-(n-1)^{-1/2})^{n-3}}\leq 3(n-1)^{-1/2}
\end{equation*}
if $n\geq 4$.
\end{proof}

\subsubsection{Bounds for \texorpdfstring{$q\ll 1$}{ q<<1}}
The strategy to bound the covariance contribution for the type 5 and type 6 terms when $q\leq 1-(n-1)^{-1/2}$ is as follows. Consider the sum over $i,j$ and break the sum up into two parts: a band around the diagonal $i=j$ of size $m$ and the rest. Then use two different strategies and optimize over the parameter $m$.

First, begin with the lemma to control the diagonal.
\begin{lemma}
\label{lem: type 5 diag}
Fix some $0\leq m\leq n-1$ and $q\leq 1-(n-1)^{-1/2}$ with $n\geq 4$. Then
\begin{equation}
\label{eq: diagonal sum}
    \sum _{|i-j|< m}\Prob\left(\substack{w(i+1)-w(i)=1\\w(j+1)-w(j)=1}\right)\leq 216m(n-1)(1-q)^2+3(n-1).
\end{equation}
\end{lemma}
\begin{proof}
First, sum over the terms with $|i-j|\leq 1$ giving $3(n-1)$. Then using Lemma \ref{lem: prob of adj bound} to bound the terms in the sum \eqref{eq: diagonal sum} and summing over $i,j$ gives the desired result, noting that
\begin{equation*}
    \frac{1}{1-q^{n-3}}\leq \frac{1}{1-(1-(n-1)^{-1/2})^{n-3}}\leq 3
\end{equation*}
if $n\geq 4$.
\end{proof}

Next, deal with the remaining terms when $|i-j|\geq m$.
\begin{lemma}
\label{lem: type 5 off-diag}
Fix some $2\leq m\leq n-1$ and $q\leq 1-(n-1)^{-1/2}$. Let $A_i$ denote the event that $w(i+1)-w(i)=1$. Then
\begin{equation*}
    \sum _{|i-j|\geq m}\Cov(I_{A_i},I_{A_j})\leq \frac{(200m+1000)q^{m-13}m(n-1)}{1-q}+96q^{m-1}(n-1).
\end{equation*}
\end{lemma}
\begin{proof}
Note that
\begin{equation*}
    \Cov(I_{A_i},I_{A_j})=\sum _{k,l}\Prob\left(\substack{w(i)=k,w(i+1)=k+1\\w(j)=l,w(j+1)=l+1}\right)-\Prob\left(\substack{w(i)=k\\w(i+1)=k+1}\right)\Prob\left(\substack{w(j)=l\\w(j+1)=l+1}\right).
\end{equation*}
Assume that $i<j-1$. Then
\begin{equation*}
    \Prob\left(\substack{w(i)=k,w(i+1)=k+1\\w(j)=l,w(j+1)=l+1}\right)=\sum _{C}q^{l(C)}\frac{[j-i-2]_q!}{[n]_q!}
\end{equation*}
where the sum is over assignments $C$ of numbers to indices less than $i$ and greater than $j+1$ (from $[n]\setminus \{k,k+1,l,l+1\}$) and $l(C)$ denotes the number of inversions caused by this assignment. More formally,
\begin{equation*}
	C:\{1,\dotsc,i-1,j+1,\dotsc,n\}\to [n]\setminus \{k,k+1,l,l+1\}
\end{equation*}
is an injective function and $l(C)=l(w_C)$ where $w_C$ is the permutation given by extending $C$ by taking $w(i)=k$, $w(i+1)=k+1$, $w(j)=l$, $w(j+1)=l+1$, and such that for all $a,b\in \{i+2,\dotsc,j-1\}$, $w(a)<w(b)$.

Similarly,
\begin{equation*}
    \Prob\left(\substack{w(k)=i\\w(k+1)=i+1}\right)\Prob\left(\substack{w(l)=j\\w(l+1)=j+1}\right)=\sum _{C',C''}q^{l(C')+l(C'')}\frac{[n-i-1]_q![j-1]_q!}{[n]_q!^2}
\end{equation*}
where the sum is over $C'$ assignments of numbers to indices less than $i$ from $[n]\setminus \{k,k+1\}$ and $l(C')$ denotes the number of inversions caused by this assignment (in the same sense as defined above), and $C''$ is over assignments of numbers to indices right of $j+1$ from $[n]\setminus \{l,l+1\}$ and $l(C'')$ is defined similarly.

Notice that every assignment $C$ induces assignments $C'$ and $C''$, and if assignments $C'$ and $C''$ are such that they share no numbers in common and $C'$ does not use $l$ or $l+1$ and $C''$ does not use $k$ or $k+1$, then they induce an assignment $C$. Moreover, if $C$ satisfies $C(a)<C(b)$ for all $a<i$ and $b>j+1$, then if $C'$ and $C''$ are the induced assignments, $l(C)=l(C')+l(C'')$.

Then the sum can be further split into three terms. There is a sum over the $C$ inducing $C', C''$ such that $l(C)=l(C')+l(C'')$, a sum over the $C$ inducing $C', C''$ where $l(C)\neq l(C')+l(C'')$ and a sum over $C',C''$ that do not induce an assignment $C$, giving
\begin{equation}
\label{eq: 3 terms}
\begin{split}
    &\sum _{C}q^{l(C)}\left(\frac{[j-i-2]_q![n]_q!-[n-i-1]_q![j-1]_q!}{[n]_q!^2}\right)
    \\&\qquad +\sum _{C}\left(\frac{q^{l(C)}[j-i-2]_q![n]_q!-q^{l'(C)}[n-i-1]_q![j-1]_q!}{[n]_q!^2}\right)
    \\&\qquad-\sum _{C,C'}q^{l(C)+l(C')}\frac{[n-i-1]_q![j-1]_q!}{[n]_q!^2}.
\end{split}
\end{equation}

The first sum in \eqref{eq: 3 terms} gives
\begin{equation*}
\begin{split}
	&\sum _{C}q^{l(C)}\left(\frac{[j-i-2]_q![n]_q!-[n-i-1]_q![j-1]_q!}{[n]_q!^2}\right)
	\\=&\sum_{C}q^{l(C)}\frac{[j-i-2]_q!}{[n]_q!}\left(1-\frac{[n-i-1]_q![j-1]_q!}{[n]_q![j-i-2]_q!}\right).
\end{split}
\end{equation*}
where the sum is over $C$ such that $C(a)<C(b)$ for all $a<i$ and $b>j+1$. Note that
\begin{equation*}
\begin{split}
    \sum_{C}q^{l(C)}\frac{[j-i-2]_q!}{[n]_q!}&\leq \Prob\left(\substack{w(i)=k,w(i+1)=k+1\\w(j)=l,w(j+1)=l+1}\right)
\end{split}
\end{equation*}
and
\begin{equation*}
\begin{split}
    1-\frac{[n-i-1]_q![j-1]_q!}{[n]_q![j-i-2]_q!}&=1-\frac{(1-q^{j-1})\dotsm(1-q^{j-i-1})}{(1-q^n)\dotsm (1-q^{n-i})}
    \\&\leq 1- (1-q^{j-1})\dotsm(1-q^{j-i-1})
    \\&\leq \frac{q^{j-i-1}(1-q^{i+1})}{1-q}
\end{split}
\end{equation*}
where $1-\prod (1-x_i)\leq \sum x_i$ for $0<x_i<1$, and so the first term has a bound of
\begin{equation*}
    \Prob\left(\substack{w(i)=k,w(i+1)=k+1\\w(j)=l,w(j+1)=l+1}\right)q^{|i-j|-1}
\end{equation*}
after using symmetry to consider the case $j<i$ and noting that $1-q^{i+1}\leq 1-q^n$ as $i\leq n-1$.

Summing $q^{|i-j|}$ over $i$ and $j$ (with the restriction that $|i-j|\geq m$) gives
\begin{equation*}
    \sum _{i,j}q^{|i-j|}\leq 2(n-1)q^{m-1}\frac{1-q^{n-1}}{1-q}
\end{equation*}
and combining this with Lemma \ref{lem: prob of adj bound} and the fact that
\begin{equation*}
    \frac{1-q^{n-1}}{1-q^{n-3}}\leq 4
\end{equation*}
when $n\geq 4$ gives a bound for the first term of
\begin{equation*}
    96(n-1)q^{m-1}.
\end{equation*}

For the second term of \eqref{eq: 3 terms}, first note that since only an upper bound is needed, the negative part can be thrown away. 

Now the remaining terms give the probability that $w(i)=k$, $w(i+1)=k+1$, $w(j)=l$, $w(j+1)=l+1$ and there is some $a<i$, $b>j+1$ such that $w(a)>w(b)$. To compute this, sum over the possible values of $w(a)$ and $w(b)$ with a union bound, giving
\begin{equation}
\label{eq: big sum}
    \sum _{k,l}\sum _{a<i}\sum _{b>j+1}\sum _{x>y}\Prob\left(\substack{w(i)=k, w(i+1)=k+1\\w(j)=l,w(j+1)=l+1\\w(a)=x,w(b)=y}\right).
\end{equation}
Then by Lemma \ref{lem: prob bound},
\begin{equation}
\label{eq: big prob bound}
    \Prob\left(\substack{w(i)=k, w(i+1)=k+1\\w(j)=l,w(j+1)=l+1\\w(a)=x,w(b)=y}\right)\leq \frac{q^{2|i-k|+2|j-l|+|a-x|+|b-y|-13}(1-q)^6}{(1-q^{n-1})^2(1-q^{n-5})^4}
\end{equation}
for similar reasons as in Lemma \ref{lem: prob of adj bound}.

First, use \eqref{eq: big prob bound} to sum \eqref{eq: big sum} over $k,l$ as above, giving a bound of
\begin{equation}
\label{eq: sum over k,l}
    \sum _{a<i}\sum _{b>j+1}\sum _{x>y}q^{|a-x|+|b-y|-13}\frac{4(1-q)^4}{(1-q^{n-5})^4}.
\end{equation}
Now compute the sum in \eqref{eq: sum over k,l} over $x>y$ over three regions. In the region $y<x<a$, $|a-x|+|b-y|=a-x+b-y$ and so
\begin{equation}
\label{eq: region 1}
\begin{split}
    \sum _{a>x>y}q^{a-x+b-y-13}&\leq\sum _{a>x, a>y}q^{a-x+b-y-13}
    \\&\leq  \frac{q^{b-a-13}(1-q^a)^2}{(1-q)^2}.
\end{split}
\end{equation}
Similarly, over the region $b<y<x$, $|a-x|+|b-y|=x-a+y-b$ giving
\begin{equation}
\label{eq: region 2}
    \sum _{b<y<x}q^{x-a+y-b-13}\leq \frac{q^{b-a-13}(1-q^{n-b})^2}{(1-q)^2}.
\end{equation}
Finally, the middle region gives
\begin{equation}
\label{eq: middle region}
\begin{split}
    \sum _{a<x, y<b, x>y}q^{x-a+b-y-13}&\leq q^{b-a-13}\sum _{k}(b-a+k)q^k
    \\&\leq (b-a)q^{b-a-13}\frac{1-q^n}{1-q}+\frac{q^{b-a-13}}{(1-q)^2}.
\end{split}
\end{equation}
In total, combining the bounds \eqref{eq: region 1}, \eqref{eq: region 2} and \eqref{eq: middle region}, the bound
\begin{equation}
\label{eq: sum over x,y}
    \sum _{x>y}q^{|a-x|+|b-y|-13}\leq (b-a)q^{b-a-13}\frac{1-q^n}{1-q}+3\frac{q^{b-a-13}}{(1-q)^2}
\end{equation}
is obtained.

Now
\begin{equation}
\label{eq: sum over a,b 1}
\begin{split}
    \sum _{a<i, b>j}(b-a)q^{b-a}&\leq \sum _{a=1}^{i-1}\sum _{b=j+1}^\infty(b-a)q^{b-a}
    \\&=\frac{(j-i)q^{j-i+2}-(j+1)q^{j+1}}{(1-q)^2}+\frac{2q^{j-i+2}-2q^{j+2}}{(1-q)^3}
    \\&\leq \frac{(j-i)q^{j-i}}{(1-q)^2}+\frac{2q^{j-i}}{(1-q)^3}
\end{split}
\end{equation}
and
\begin{equation}
\label{eq: sum over a,b, 2}
\begin{split}
    \sum _{a<i, b>j}q^{b-a}&\leq \frac{q^{j-i}(1-q^n)^2}{(1-q)^2}
\end{split}
\end{equation}
and so combining \eqref{eq: sum over a,b 1} and \eqref{eq: sum over a,b, 2} gives a bound of
\begin{equation}
\label{eq: sum over a, b total}
    \frac{(j-i)q^{j-i-13}}{(1-q)^3}+\frac{2q^{j-i-13}(1-q^n)}{(1-q)^4}+3\frac{q^{j-i-13}(1-q^n)^2}{(1-q)^4}
\end{equation}
for \eqref{eq: sum over x,y}. Finally, \eqref{eq: sum over a, b total} needs to be summed over $i,j$ with $|i-j|\geq m$, which by symmetry can be reduced to a sum over $j-i\geq m$. Now
\begin{equation}
\label{eq: sum over i, j 1}
\begin{split}
    \sum _{j-i\geq m}{(j-i)q^{j-i}}&\leq (n-1)\sum _{c=m}^\infty cq^c\leq \frac{m(n-1)q^{m}}{(1-q)^2}
\end{split}
\end{equation}
and
\begin{equation}
\label{eq: sum over i, j 2}
    \sum _{j-i\geq m}q^{j-i}\leq (n-1)\sum _{c=m}^n q^{c} \leq nq^m\frac{1-q^n}{1-q},
\end{equation}
and so using \eqref{eq: sum over i, j 1} and \eqref{eq: sum over i, j 2} to bound \eqref{eq: sum over a, b total} gives
\begin{equation}
\label{eq: sum over everything}
\begin{split}
    &\sum _{|i-j|\geq m}\sum _{a<i}\sum _{b>j+1}\sum _{x>y}q^{|a-x|+|b-y|-13}
    \\\leq &\frac{2m(n-1)q^{m-13}}{(1-q)^5}+\frac{10(n-1)q^{m-13}(1-q^n)^2}{(1-q)^5}.
\end{split}
\end{equation}
Then the sum of \eqref{eq: big sum} over $|i-j|\geq m$ is bounded by
\begin{equation}
    \frac{(200m+1000)(n-1)q^{m-13}}{1-q},
\end{equation}
noting that
\begin{equation*}
    \frac{1}{(1-q^{n-5})}\leq \frac{1}{(1-(1-(n-1)^{-1/2})^{n-5}}\leq 5^{1/2}
\end{equation*}
when $q\leq 1-(n-1)^{-1/2}$ and $n\geq 6$.

Finally, the third term of \eqref{eq: 3 terms} can be discarded as it's negative.
\end{proof}

\subsubsection{Unconditional bounds}
Finally, optimize over the parameter $m$ to obtain the desired bounds.
\begin{lemma}
\label{lem: type 5 total}
The type 5 terms are bounded by $6507(n-1)$.
\end{lemma}
\begin{proof}
When $q\geq 1-(n-1)^{-1/2}$, Lemma \ref{lem: type 5 large} gives the desired bound, so assume $q\leq 1-(n-1)^{-1/2}$.

Note that $\des(w^{-1})-\des((w_{i}^*)^{-1})$ is $-1$ if $w(i+1)-w(i)=1$ and $0$ otherwise, so if $A_i$ denotes this event, the total contribution is given by
\begin{equation*}
    \sum _{i,j}\Cov(I_{A_i}, I_{A_j}).
\end{equation*}
Now picking $0\leq m\leq n$, Lemmas \ref{lem: type 5 off-diag} and \ref{lem: type 5 diag} gives a bound of
\begin{equation}
\label{eq: unoptimized bound}
    \left(99+216m(1-q)^2+\frac{(200m+1000)q^{m-13}}{1-q}\right)(n-1).
\end{equation}
Take
\begin{equation*}
    m=3\frac{\log(1-q)}{\log(q)}+13.
\end{equation*}
Now
\begin{equation*}
    m\leq 3(n-1)^{1/2}\log((n-1)^{1/2})+13 \leq n-1
\end{equation*}
if $n\geq 65$ since $q\leq 1-(n-1)^{-1/2}$. If $n\leq 65$, the desired bound holds, so it can be assumed that $0\leq m\leq n-1$.

By the choice of $m$, $q^m=(1-q)^3q^{13}$, and substituting this into \eqref{eq: unoptimized bound} gives a bound of
\begin{equation*}
    (n-1)(99+416m(1-q)^2+1000)\leq 6507(n-1)
\end{equation*}
as $m(1-q)^2\leq 13$ when $q\in (0,1)$.
\end{proof}

\begin{lemma}
\label{lem: type 6}
The type 6 terms are bounded by $6507(n-1)$.
\end{lemma}
\begin{proof}
The proof proceeds similarly to Lemma \ref{lem: type 5 total}. Let $A$ be the event that $i$ and $i+1$ are adjacent in $w$ and $i$ appears before $i+1$, and $B$ be the event that $w(i+1)-w(i)=1$. As in Lemma \ref{lem: type 5 large}, $\des(w)-\des(w_{-i}^*)$ and $\des(w^{-1})-\des((w_i^*)^{-1})$ are equal to $-I_A$ and $-I_B$ respectively. Then the terms of type 6 give a contribution of
\begin{equation*}
    \sum _{i,j}\sum _{k,l}\Prob\left(\substack{w(k)=i, w(k+1)=i+1\\w(j)=l, w(j+1)=l+1}\right)-\Prob\left(\substack{w(k)=i\\w(k+1)=i+1}\right)\Prob\left(\substack{w(j)=l\\w(j+1)=l+1}\right)
\end{equation*}
which is the same sum as appears in Lemma \ref{lem: type 5 total} after reordering the summation and so the same bound is obtained.
\end{proof}

\section{Limit theorems}
\label{sec: main thms}
\subsection{Proof of Theorem \ref{thm: main theorem 1}}
\begin{proof}[Proof of Theorem \ref{thm: main theorem 1}]
The proof is an easy consequence of the computations done in Lemmas \ref{lem: type 1}, \ref{lem: type 2}, \ref{lem: type 3}, \ref{lem: type 4}, \ref{lem: type 5 total} and \ref{lem: type 6}. 

By Theorem \ref{thm: size-bias steins}, it suffices to bound
\begin{equation*}
    2\frac{\mu}{\sigma^2}\sqrt{\Var\E(\des(w)+\des(w^{-1})-\des(w^*)-\des((w^*)^{-1})|w)}
\end{equation*}
and
\begin{equation*}
    \frac{\mu}{\sigma^3}\E(\des(w)+\des(w^{-1})-\des(w^*)-\des((w^*)^{-1}))^2.
\end{equation*}

Now by Proposition \ref{prop: mean var calc}
\begin{equation*}
    \mu=\frac{2q(n-1)}{1+q}
\end{equation*}
and
\begin{equation*}
    \sigma^2\geq \frac{q(n-1)(1-q+q^2)}{(1+q)^2(1+q+q^2)}.
\end{equation*}
Then
\begin{equation*}
    \frac{\mu}{\sigma^3}\leq 24\sqrt{3}q^{-\frac{1}{2}}(n-1)^{-\frac{1}{2}}
\end{equation*}
and as $|\des(w)+\des(w^{-1})-\des(w^*)-\des((w^*)^{-1})|\leq 2$, the second term is bounded by $167q^{-1/2}(n-1)^{-1/2}$

For the first term, combining Lemmas \ref{lem: type 1}, \ref{lem: type 2}, \ref{lem: type 3}, \ref{lem: type 4}, \ref{lem: type 5 large}, \ref{lem: type 5 total} and \ref{lem: type 6} with the multiplicities in Table \ref{tab:term types} and the prefactor of $(2(n-1))^{-1}$ from the conditional expectation gives a variance bound of
\begin{equation*}
    6847(n-1)^{-1}.
\end{equation*}
Then
\begin{equation*}
    \frac{\mu}{\sigma^2}\leq \frac{2(1+q)(1+q+q^2)}{1-q+q^2}\leq 2
\end{equation*}
and so the first term is bounded by
\begin{equation*}
    331(n-1)^{-\frac{1}{2}}.
\end{equation*}

Finally, if $q>1$, then $w^{rev}$ is distributed as Mallows with parameter $q^{-1}$ by Proposition \ref{prop: distr of rev}. This corresponds to negating
\begin{equation*}
    \frac{\des(w)+\des(w^{-1})-\mu}{\sigma}
\end{equation*}
by Proposition \ref{prop: desc of rev} and so the same bound holds for $q>1$, except $q$ is replaced with $q^{-1}$.
\end{proof}

\subsection{Proof of Theorem \ref{thm: bivariate limit}}
With the computation of the asymptotic correlation, the joint central limit theorem for $\des(w)$ and $\des(w^{-1})$ follows easily from the Cram\'er-Wold theorem.
\begin{proof}[Proof of Theorem \ref{thm: bivariate limit}]
Note that the coupling in Proposition \ref{prop: size-bias coupling} can easily be adapted to the random variable $a\des(w)+b\des(w^{-1})$ for $a,b>0$. Furthermore, with this coupling, the covariance bounds obtained in Section \ref{sec: cov bounds} still hold with some constants depending on $a,b$. Then if $Z$ is a standard normal random variable, by Theorem \ref{thm: size-bias steins},
\begin{equation*}
    \frac{a\des(w)+b\des(w^{-1})-\frac{(a+b)(n-1)q}{1+q}}{\Var(a\des(w)+b\des(w^{-1}))}\xrightarrow{d} Z
\end{equation*}
as long as $qn\to \infty$.

Now assume that the limit
\begin{equation*}
    \rho=\lim _{n\to \infty}\frac{\Cov(\des(w),\des(w^{-1}))}{\Var(\des(w))}
\end{equation*}
exists. Then
\begin{equation*}
    \frac{a\des(w)+b\des(w^{-1})-\frac{(a+b)(n-1)q}{1+q}}{\Var(\des(w))}\xrightarrow{d} (a^2+b^2+2ab\rho)Z,
\end{equation*}
and so by the Cram\'er-Wold theorem,
\begin{equation*}
    \left(\frac{\des(w_n)-\mu_n}{\sigma_n},\frac{\des(w_n^{-1})-\mu_n}{\sigma_n}\right)\xrightarrow{d}(Z_1,Z_2)
\end{equation*}
where
\begin{equation*}
    (Z_1,Z_2)\sim N\left(0,\left(\begin{array}{cc}
         1&\rho  \\
         \rho&1 
    \end{array}\right)\right).
\end{equation*}

If $0<q<1$ is fixed, the limit exists by Proposition \ref{prop: asymp cov}. The bound on $\Cov(\des(w),\des(w^{-1}))$ in Proposition \ref{prop: mean var calc} shows that $0<\rho$, and $\rho<1$ as discussed in Remark \ref{rmk: asymp cov}.

If $q_n\to 0$ or $q_n\to 1$, then the limit exists and $\rho=0$ or $\rho=1$ respectively by Corollary \ref{cor: asymp cor 1 0}.

Finally, by Proposition \ref{prop: desc of rev} and Proposition \ref{prop: distr of rev}, the result for $q\leq 1$ immediately extends to all $q$.
\end{proof}

\subsection{Poisson limit theorem}
In this section, it is shown that in the regime $qn\to \lambda$, $\des(w)+\des(w^{-1})$ has Poisson behaviour. This is much easier than the central limit theorem because in this regime $\des(w)$ and $\des(w^{-1})$ are asymptotically equal and so it suffices to study $\des(w)$.

\begin{proposition}
\label{prop: don't move too far}
Let $w\in S_n$ be Mallows distributed. Then
\begin{equation*}
    \Prob(|w(i)-i|>1)\leq q^2
\end{equation*}
\end{proposition}
\begin{proof}
Note that
\begin{equation*}
    \Prob(|w(i)-i|>1)=\sum _{j\not\in \{i-1,i,i+1\}}\Prob(w(i)=j).
\end{equation*}
Now by Lemma \ref{lem: prob bound}, this is bounded by
\begin{equation*}
    \sum _{j\not\in \{i-1,i,i+1\}}\frac{q^{|j-i|}(1-q)}{(1-q^n)}\leq 2q^2.
\end{equation*}
\end{proof}

\begin{lemma}
\label{lem: cond for equality}
Suppose $w\in S_n$ satisfies $|w(i)-i|\leq 1$ for all $i$. Then $\des(w)=\des(w^{-1})$, and in fact $\des_i(w)=\des_i(w^{-1})$ for all $i$.
\end{lemma}
\begin{proof}
First, note that if $|w(i)-i|\leq 1$ for all $i$, then the same is true for $w^{-1}$. If $|w(i)-i|\leq 1$ for all $i$, then the only way to have a descent at $i$ is to have $w(i)=i+1$ and $w(i+1)=i$. But then $w^{-1}(i)=i+1$ and $w^{-1}(i+1)=i$ so $w^{-1}$ also has a descent at $i$. By symmetry, $w$ can have a descent at $i$ if and only if $w^{-1}$ does.
\end{proof}

These two results imply that $\des(w)=\des(w^{-1})$ with high probability and reduces the study of $\des(w)+\des(w^{-1})$ to simply studying $2\des(w)$.
\begin{corollary}
\label{cor: des and ides equal}
If $w\in S_n$ is Mallows distributed, then
\begin{equation*}
    \Prob\left(\des_i(w)=\des_i(w^{-1})\textnormal{ for all }i\right)\geq 1-2q^2n.
\end{equation*}
\end{corollary}
\begin{proof}
By Proposition \ref{prop: don't move too far}, a union bound gives that $|w(i)-i|\leq 1$ with probability at least $1-q^2n$. On this event, $\des(w)=\des(w^{-1})$ by Lemma \ref{lem: cond for equality}.
\end{proof}
Let $d_{TV}$ denote the total variation distance between probability measures.
\begin{corollary}
If $w\in S_n$ is Mallows distributed, then
\begin{equation*}
    d_{TV}(\des(w)+\des(w^{-1}), 2\des(w))\leq 2q^2n.
\end{equation*}
\end{corollary}
\begin{proof}
Note that
\begin{equation*}
\begin{split}
    d_{TV}(\des(w)+\des(w^{-1}),2\des(w))&\leq \Prob(\des(w)\neq \des(w^{-1}))
    \\&\leq 2q^2n.
\end{split}
\end{equation*}
\end{proof}

The next result bounds the total variation distance between $\des(w)$ and a Poisson random variable. It follows from a special case of Theorem 4.13 in \cite{CdS17}. The size-bias coupling given by Proposition \ref{prop: size-bias coupling} could be used to give another proof of this, using Stein's method for Poisson approximation with size-bias couplings (see \cite[Theorem 4.13]{R11b} for example).
\begin{theorem}[{\cite[Theorem 3.8]{CdS17}}]
Let $w$ be Mallows distributed and let $\lambda=\E(\des(w))$. Let $N$ denote a Poisson random variable of mean $\lambda$. Then
\begin{equation*}
    d_{TV}(\des(w), N)\leq 10nq^2.
\end{equation*}
\end{theorem}

The following result on Poisson approximation for $\des(w)+\des(w^{-1})$ follows easily from the triangle inequality.
\begin{proposition}
\label{prop: poisson behaviour}
Fix $0<q\leq 1$ and let $w\in S_n$ be Mallows distributed. Let $\lambda=(n-1)q/(1+q)$ and let $N$ be Poisson with mean $\lambda$. Then
\begin{equation*}
    d_{TV}(\des(w)+\des(w^{-1}), 2N)\leq 12q^2n.
\end{equation*}
\end{proposition}

\section*{Acknowledgements}
This research was supported in part by NSERC. The author would like to thank Persi Diaconis for helpful discussions, Andrea Ottolini for a careful reading of the manuscript, Jim Pitman for pointing out some useful references and suggesting the concentration bounds, and Wenpin Tang for pointing out some useful references.

\bibliography{bibliography}{}

\providecommand{\bysame}{\leavevmode\hbox to3em{\hrulefill}\thinspace}
\providecommand{\MR}{\relax\ifhmode\unskip\space\fi MR }
\providecommand{\MRhref}[2]{%
  \href{http://www.ams.org/mathscinet-getitem?mr=#1}{#2}
}
\providecommand{\href}[2]{#2}
\begin{thebibliography}{10}

\bibitem{AF}
David Aldous and James~Allen Fill, \emph{Reversible markov chains and random
  walks on graphs}, 2002, Unfinished monograph, recompiled 2014, available at
  \url{http://www.stat.berkeley.edu/~aldous/RWG/book.html}.

\bibitem{AHHL18}
Omer Angel, Alexander~E. Holroyd, Tom Hutchcroft, and Avi Levy, \emph{Mallows
  permutations as stable matchings}, 2018.

\bibitem{BRS89}
P.~Baldi, Y.~Rinott, and C.~Stein, \emph{A normal approximation for the number
  of local maxima of a random function on a graph}, Probability, statistics,
  and mathematics, Academic Press, Boston, MA, 1989, pp.~59--81.

\bibitem{BB17}
Riddhipratim Basu and Nayantara Bhatnagar, \emph{Limit theorems for longest
  monotone subsequences in random {M}allows permutations}, Ann. Inst. Henri
  Poincar\'{e} Probab. Stat. \textbf{53} (2017), no.~4, 1934--1951.

\bibitem{BP15}
Nayantara Bhatnagar and Ron Peled, \emph{Lengths of monotone subsequences in a
  {M}allows permutation}, Probab. Theory Related Fields \textbf{161} (2015),
  no.~3-4, 719--780.

\bibitem{BB05}
Anders Bj\"{o}rner and Francesco Brenti, \emph{Combinatorics of {C}oxeter
  groups}, Graduate Texts in Mathematics, vol. 231, Springer, New York, 2005.

\bibitem{BDF10}
Alexei Borodin, Persi Diaconis, and Jason Fulman, \emph{On adding a list of
  numbers (and other one-dependent determinantal processes)}, Bull. Amer. Math.
  Soc. (N.S.) \textbf{47} (2010), no.~4, 639--670.

\bibitem{BR19}
Benjamin Brück and Frank Röttger, \emph{A central limit theorem for the
  two-sided descent statistic on coxeter groups}, 2019.

\bibitem{B20}
Alexey Bufetov, \emph{Interacting particle systems and random walks on hecke
  algebras}, 2020.

\bibitem{CD17}
Sourav Chatterjee and Persi Diaconis, \emph{A central limit theorem for a new
  statistic on permutations}, Indian J. Pure Appl. Math. \textbf{48} (2017),
  no.~4, 561--573.

\bibitem{CV07}
Mark Conger and D.~Viswanath, \emph{Normal approximations for descents and
  inversions of permutations of multisets}, J. Theoret. Probab. \textbf{20}
  (2007), no.~2, 309--325.

\bibitem{CGJ18}
Nicholas Cook, Larry Goldstein, and Tobias Johnson, \emph{Size biased couplings
  and the spectral gap for random regular graphs}, Ann. Probab. \textbf{46}
  (2018), no.~1, 72--125.

\bibitem{CdS17}
Harry Crane and Stephen DeSalvo, \emph{Pattern avoidance for random
  permutations}, Discrete Math. Theor. Comput. Sci. \textbf{19} (2017), no.~2,
  Paper No. 13, 24.

\bibitem{CdSE18}
Harry Crane, Stephen DeSalvo, and Sergi Elizalde, \emph{The probability of
  avoiding consecutive patterns in the {M}allows distribution}, Random
  Structures Algorithms \textbf{53} (2018), no.~3, 417--447.

\bibitem{DR00}
Persi Diaconis and Arun Ram, \emph{Analysis of systematic scan {M}etropolis
  algorithms using {I}wahori-{H}ecke algebra techniques}, Michigan Math. J.
  \textbf{48} (2000), 157--190.

\bibitem{E16}
Sergi Elizalde, \emph{A survey of consecutive patterns in permutations}, Recent
  trends in combinatorics, IMA Vol. Math. Appl., vol. 159, Springer, [Cham],
  2016, pp.~601--618. \MR{3526425}

\bibitem{FGHHPRT20}
Xiao Fang, Han~Liang Gan, Susan Holmes, Haiyan Huang, Erol Peköz, Adrian
  Röllin, and Wenpin Tang, \emph{Arcsine laws for random walks generated from
  random permutations with applications to genomics}, 2020.

\bibitem{F71}
William Feller, \emph{An introduction to probability theory and its
  applications. {V}ol. {II}}, Second edition, John Wiley \& Sons, Inc., New
  York-London-Sydney, 1971.

\bibitem{F18}
Valentin F\'{e}ray, \emph{Weighted dependency graphs}, Electron. J. Probab.
  \textbf{23} (2018), Paper No. 93, 65.

\bibitem{F19}
\bysame, \emph{On the central limit theorem for the two-sided descent
  statistics in coxeter groups}, 2019.

\bibitem{F20}
\bysame, \emph{Central limit theorems for patterns in multiset permutations and
  set partitions}, Ann. Appl. Probab. \textbf{30} (2020), no.~1, 287--323.

\bibitem{F98}
Jason Fulman, \emph{The distribution of descents in fixed conjugacy classes of
  the symmetric groups}, J. Combin. Theory Ser. A \textbf{84} (1998), no.~2,
  171--180.

\bibitem{F04}
\bysame, \emph{Stein's method and non-reversible {M}arkov chains}, Stein's
  method: expository lectures and applications, IMS Lecture Notes Monogr. Ser.,
  vol.~46, Inst. Math. Statist., Beachwood, OH, 2004, pp.~69--77.

\bibitem{FKL19}
Jason Fulman, Gene~B. Kim, and Sangchul Lee, \emph{Central limit theorem for
  peaks of a random permutation in a fixed conjugacy class of {$S_n$}}, 2019.

\bibitem{FKLP19}
Jason Fulman, Gene~B. Kim, Sangchul Lee, and T.~Kyle Petersen, \emph{On the
  joint distribution of descents and signs of permutations}, 2019.

\bibitem{G13}
Robert~G. Gallager, \emph{Stochastic processes: Theory for applications},
  Cambridge University Press, Cambridge, 2013.

\bibitem{GP18}
Alexey Gladkich and Ron Peled, \emph{On the cycle structure of {M}allows
  permutations}, Ann. Probab. \textbf{46} (2018), no.~2, 1114--1169.

\bibitem{GO10}
Alexander Gnedin and Grigori Olshanski, \emph{{$q$}-exchangeability via
  quasi-invariance}, Ann. Probab. \textbf{38} (2010), no.~6, 2103--2135.

\bibitem{GO12}
\bysame, \emph{The two-sided infinite extension of the {M}allows model for
  random permutations}, Adv. in Appl. Math. \textbf{48} (2012), no.~5,
  615--639.

\bibitem{G05}
Larry Goldstein, \emph{Berry-{E}sseen bounds for combinatorial central limit
  theorems and pattern occurrences, using zero and size biasing}, J. Appl.
  Probab. \textbf{42} (2005), no.~3, 661--683.

\bibitem{GR96}
Larry Goldstein and Yosef Rinott, \emph{Multivariate normal approximations by
  {S}tein's method and size bias couplings}, J. Appl. Probab. \textbf{33}
  (1996), no.~1, 1--17.

\bibitem{HHL20}
Alexander~E. Holroyd, Tom Hutchcroft, and Avi Levy, \emph{Mallows permutations
  and finite dependence}, Ann. Probab. \textbf{48} (2020), no.~1, 343--379.

\bibitem{KS20}
Thomas Kahle and Christian Stump, \emph{Counting inversions and descents of
  random elements in finite {C}oxeter groups}, Math. Comp. \textbf{89} (2020),
  no.~321, 437--464.

\bibitem{KL18}
Gene~B. Kim and Sangchul Lee, \emph{A central limit theorem for descents and
  major indices in fixed conjugacy classes of {$S_n$}}, 2018.

\bibitem{KL20}
\bysame, \emph{Central limit theorem for descents in conjugacy classes of
  {$S_n$}}, J. Combin. Theory Ser. A \textbf{169} (2020), 105123.

\bibitem{LL19}
Cyril Labb\'{e} and Hubert Lacoin, \emph{Cutoff phenomenon for the asymmetric
  simple exclusion process and the biased card shuffling}, Ann. Probab.
  \textbf{47} (2019), no.~3, 1541--1586.

\bibitem{M57}
C.~L. Mallows, \emph{Non-null ranking models. {I}}, Biometrika \textbf{44}
  (1957), 114--130.

\bibitem{MS13}
Carl Mueller and Shannon Starr, \emph{The length of the longest increasing
  subsequence of a random {M}allows permutation}, J. Theoret. Probab.
  \textbf{26} (2013), no.~2, 514--540.

\bibitem{M16b}
Sumit Mukherjee, \emph{Estimation in exponential families on permutations},
  Ann. Statist. \textbf{44} (2016), no.~2, 853--875.

\bibitem{M16a}
\bysame, \emph{Fixed points and cycle structure of random permutations},
  Electron. J. Probab. \textbf{21} (2016), Paper No. 40, 18.

\bibitem{O19}
Alperen~Y. \"{O}zdemir, \emph{Martingales and descent statistics}, 2019.

\bibitem{PT19}
Jim Pitman and Wenpin Tang, \emph{Regenerative random permutations of
  integers}, Ann. Probab. \textbf{47} (2019), no.~3, 1378--1416.

\bibitem{R11b}
Nathan Ross, \emph{Fundamentals of {S}tein's method}, Probab. Surv. \textbf{8}
  (2011), 210--293.

\bibitem{R18}
Frank Röttger, \emph{Asymptotics of a locally dependent statistic on finite
  reflection groups}, 2018.

\bibitem{T19}
Wenpin Tang, \emph{Mallows ranking models: maximum likelihood estimate and
  regeneration}, Proceedings of the 36th International Conference on Machine
  Learning (Long Beach, California, USA) (Kamalika Chaudhuri and Ruslan
  Salakhutdinov, eds.), Proceedings of Machine Learning Research, vol.~97,
  PMLR, 09--15 Jun 2019, pp.~6125--6134.

\bibitem{V96}
V.~A. Vatutin, \emph{The numbers of ascending segments in a random permutation
  and in one inverse to it are asymptotically independent}, Diskret. Mat.
  \textbf{8} (1996), no.~1, 41--51.

\end{thebibliography}
\bibliographystyle{amsplain}

\end{document}